\pdfoutput=1 

\documentclass[11pt]{amsart}

\usepackage{amssymb,amsthm,enumitem,colonequals,tikz-cd,microtype}
\usepackage[normalem]{ulem} 

\usepackage[osf]{Baskervaldx}
\usepackage[baskervaldx]{newtxmath}
\usepackage[cal=boondoxo]{mathalfa} 

\usepackage[top=3.75cm, bottom=3cm, left=3.5cm, right=3.5cm]{geometry}

\usepackage{xcolor}  
\colorlet{darkblue}{blue!55!black}
\colorlet{darkcyan}{cyan!50!black}
\colorlet{darkgreen}{green!60!black}

\PassOptionsToPackage{hyphens}{url}
\usepackage{hyperref}
\hypersetup{
    colorlinks=true,
    linkcolor=darkblue,
    urlcolor=darkcyan,
    citecolor=darkgreen,
}

\def\eqref#1{\textcolor{darkblue}{(\ref{#1})}}

\usepackage[nameinlink]{cleveref} 
\Crefformat{section}{#2\S#1#3}
\Crefmultiformat{section}{#2\S\S#1#3}{ and~#2#1#3}{, #2#1#3}{, and~#2#1#3}

\usepackage[pagewise]{lineno}
\let\oldequation\equation
\let\oldendequation\endequation
\renewenvironment{equation}{\linenomathNonumbers\oldequation}{\oldendequation\endlinenomath}
\expandafter\let\expandafter\oldequationstar\csname equation*\endcsname
\expandafter\let\expandafter\oldendequationstar\csname endequation*\endcsname
\renewenvironment{equation*}{\linenomathNonumbers\oldequationstar}{\oldendequationstar\endlinenomath}
\let\oldalign\align
\let\oldendalign\endalign
\renewenvironment{align}{\linenomathNonumbers\oldalign}{\oldendalign\endlinenomath}
\expandafter\let\expandafter\oldalignstar\csname align*\endcsname
\expandafter\let\expandafter\oldendalignstar\csname endalign*\endcsname
\renewenvironment{align*}{\linenomathNonumbers\oldalignstar}{\oldendalignstar\endlinenomath}

\theoremstyle{plain}
\newtheorem{theorem}{Theorem}[section]
\newtheorem{lemma}[theorem]{Lemma}
\newtheorem{corollary}[theorem]{Corollary}
\newtheorem{proposition}[theorem]{Proposition}

\theoremstyle{definition}

\newtheorem{example}[theorem]{Example}
\newtheorem{remark}[theorem]{Remark}
\newtheorem{notation}[theorem]{Notation}
\newtheorem{question}[theorem]{Question}
\newtheorem{placeholder}[theorem]{Placeholder}

\newtheorem*{ack}{Acknowledgments}

\AddToHook{env/conjecture/begin}{\crefalias{theorem}{conjecture}}
\AddToHook{env/lemma/begin}{\crefalias{theorem}{lemma}}
\AddToHook{env/corollary/begin}{\crefalias{theorem}{corollary}}
\AddToHook{env/proposition/begin}{\crefalias{theorem}{proposition}}
\AddToHook{env/definition/begin}{\crefalias{theorem}{definition}}
\AddToHook{env/remark/begin}{\crefalias{theorem}{remark}}
\AddToHook{env/example/begin}{\crefalias{theorem}{example}}
\AddToHook{env/notation/begin}{\crefalias{theorem}{notation}}
\AddToHook{env/placeholder/begin}{\crefalias{theorem}{Placeholder}}

\numberwithin{equation}{section}
\numberwithin{theorem}{section}

\title[Dualizing complexes and $t$-structures for algebraic spaces]{Dualizing complexes and $t$-structures \\ for algebraic spaces}

\author[P.~Lank]{Pat Lank}
\address{P.~Lank,
Dipartimento di Matematica “F. Enriques”, Universit\`{a} degli Studi di Milano, Via Cesare
Saldini 50, 20133 Milano, Italy}
\email{plankmathematics@gmail.com}

\date{\today}



\begin{document}
    
\begin{abstract}
    This work is concerned with $t$-structures on Noetherian algebraic spaces admitting dualizing complexes.
    In particular, we study their behavior in the \'{e}tale topology. 
    As a consequence, we classify all tensor $t$-structures on their bounded derived category of coherent sheaves. 
\end{abstract}

\maketitle

\tableofcontents

\section{Introduction}
\label{sec:intro}

The guiding question of this article is:

\begin{question}
    \label{q:behavior}
   How do tensor $t$-structures behave in the \'{e}tale topology?
\end{question}

Two related problems naturally arise. 
First, given a tensor $t$-structure on $D_{\operatorname{qc}}$, one would like to know whether its truncation triangles are compatible with pullback along an \'{e}tale presentation.
Second, one would like to detect whether such a $t$-structure restricts to the derived category of complexes with bounded and coherent cohomology $D^b_{\operatorname{coh}}$ after passing to an \'{e}tale cover.

Neither is formal: left $t$-exactness of pullback requires control of the associated coaisle, whereas a bounded pseudocoherent complex on the source of an \'{e}tale cover need not arise from a pullback.
We address these two problems.
As a result, we classify tensor $t$-structures on $D^b_{\operatorname{coh}}$ of a Noetherian algebraic space. 

\subsection{What is known}
\label{sec:intro_what_is_known}

Dualizing complexes lie at the heart of Grothendieck duality and its applications in algebraic geometry and commutative algebra.
We use the standard \'{e}tale local definition for dualizing complexes  \cite[\href{https://stacks.math.columbia.edu/tag/0E4X}{Tag 0E4X}]{stacks-project}.
Among many things, a notable application of dualizing complexes arises in the theory of $t$-structures. 

A key structural result is due to Alonso Tarr\'{i}o, Jerem\'{i}as L\'{o}pez, and Saor\'{i}n \cite[Theorem 6.9]{AlonsoTarrio/JeremiasLopez/Saorin:2010}, who classify $t$-structures on $D^b_{\operatorname{coh}}(R)$ when $R$ is a Noetherian ring admitting a dualizing complex. 
The classification is expressed in terms of Thomason filtrations of $\operatorname{Spec}(R)$ satisfying the weak Cousin condition.
Takahashi has extended the classification on $D^b_{\operatorname{coh}}$ to the broader class of CM-excellent rings \cite{Takahashi:2023}.

The classification of $t$-structures on $D_{\operatorname{qc}}$ in general is subtle. 
Even in the affine case, a complete description of all $t$-structures on $D_{\operatorname{qc}}(X)$ is not available for $X=\operatorname{Spec}(\mathbb{Z})$ \cite{Stanley:2010}.
This motivates restricting attention to well-behaved classes of $t$-structures.

In the global setting, we imposes a tensor compatibility. 
Namely, let $X$ be a Noetherian algebraic space. 
A $t$-structure $\mathcal{A}$ on $D_{\operatorname{qc}}(X)$ is \textit{tensor} if $D^{\leq 0}_{\operatorname{qc}}(X) \otimes^\mathbf{L} \mathcal{A} \subseteq \mathcal{A}$, and a $t$-structure $(\mathcal{A},\mathcal{B})$ on $D^b_{\operatorname{coh}}(X)$ is \textit{tensor} if $D^{\leq 0}_{\operatorname{coh}}(X) \otimes^{\mathbf{L}} \mathcal{A} \subseteq \mathcal{A}$.

A further restriction is tensor compatibility with compactly generated $t$-structures. 
In the affine Noetherian setting, all $t$-structures are tensor compatible. These were classified in \cite{AlonsoTarrio/JeremiasLopez/Saorin:2010}, and later extended to more general affine schemes in \cite{Hrbek:2020}. 
In the global setting, \cite[Theorem 1.4]{Hrbek/Lank/Pizzirani:2025} establishes such a classification for quasi-compact quasi-separated algebraic spaces. 
These classification results are governed by Thomason filtrations on the underlying topological space (see \Cref{sec:prelim_Thomason}). 

Turning to $D^b_{\operatorname{coh}}$, classifications require topological input. 
In the affine Noetherian setting, if a dualizing complex exists, \cite[Theorem 6.9]{AlonsoTarrio/JeremiasLopez/Saorin:2010} shows that $t$-structures on $D^b_{\operatorname{coh}}(R)$ are classified by Thomason filtrations satisfying a weak Cousin condition.  
For CM-excellent schemes, a relative version of the classification has been obtained in \cite{Clark/Lank/ManaliRahul/Parker:2024}, building on \cite{Takahashi:2023}. 
Specifically, the tensor compatible $t$-structures on $D^b_{\operatorname{coh},Z}(X)$ correspond to Thomason filtrations satisfying the weak Cousin condition along $Z$.

\subsection{What we do}
\label{sec:intro_what_we_do}

While algebraic spaces are \'{e}tale locally schemes, the current machinery is not effective enough to handle algebraic spaces. 
The two obstacles are whether truncation functors are compatible along pullback of an \'{e}tale presentation, and how restrictability can be probed along an \'{e}tale cover.

Our main result is the following:

\begin{theorem}
    \label{thm:dualizing_complex_classification}
    Let $X$ be a Noetherian algebraic space admitting a dualizing complex. 
    Then there exists a one-to-one correspondence:
    \begin{displaymath}
        \begin{aligned}
            \{ \otimes\textrm{-aisle on } 
            & D^b_{\operatorname{coh}}(X)\} 
            \\& \iff \{ \textrm{Thomason filtrations on } |X| \textrm{ satisfying the weak Cousin condition} \}.
        \end{aligned}
    \end{displaymath}
\end{theorem}

We describe the ingredients needed for the proof.
The first is concerned with $t$-exactness for aisles and derived pullback. 

\begin{proposition}
    \label{prop:t_exactness_flat_pullback}
    Let $X$ be a Noetherian algebraic space, $s\colon U \to X$ be an \'{e}tale presentation from an affine scheme, and $\mathcal{S}\subseteq D^b_{\operatorname{coh}}(X)$. 
    Then $\mathbf{L} s^\ast$ induces a $t$-exact functor,
    \begin{displaymath}
        (D_{\operatorname{qc}}(X), \overline{\langle \mathcal{S} \rangle}^{(-\infty,0]}_{\otimes} ) \to (D_{\operatorname{qc}}(U), \overline{\langle \mathbf{L} s^\ast \mathcal{S} \rangle}^{(-\infty,0]}_{\otimes} ).
    \end{displaymath}
\end{proposition}

Right $t$-exactness is immediate from the construction of the aisles.
Left $t$-exactness is more delicate.
Its proof uses a boundedness criterion for tensor aisles on algebraic spaces, together with a `twisted' $t$-exactness result for derived pushforward between tensor aisles.
These results allow the argument to be reduced to the corresponding affine statement. 
See \Cref{sec:exactness} for details.

The second ingredient concerns controlling bounded pseudocoherent complexes under separated \'{e}tale morphisms. 

\begin{proposition}
    \label{prop:etale_verdier_localization}
    Let $s\colon U \to X$ be a separated \'{e}tale morphism of Noetherian algebraic spaces.
    Then $\mathbf{L}s^\ast \colon D^b_{\operatorname{coh}}(X) \to D^b_{\operatorname{coh}}(U)$ is essentially dense (i.e.\ for every $E\in D^b_{\operatorname{coh}}(U)$ there exists $E^\prime \in D^b_{\operatorname{coh}}(X)$ such that $E$ is a direct summand of $\mathbf{L}s^\ast E^\prime$). 
\end{proposition}

This recovers \cite[Theorems 3.5 \& 4.2]{Balmer:2016} in the case of schemes. 
The proof factors $s$, by Zariski's Main Theorem, as an open immersion followed by a finite morphism. 
A Verdier localization argument extends bounded coherent complexes across the open immersion. 
After flat base change, the graph of the open immersion determines an open-and-closed component of a certain fiber product, which produces the desired direct summand.
We believe that \Cref{prop:etale_verdier_localization} will be of independent interest to others. 

Combing \Cref{prop:t_exactness_flat_pullback,prop:etale_verdier_localization} yields a restrictability criterion.

\begin{theorem}
    \label{thm:restrict_iff_weak_cousin}
    Let $X$ be a Noetherian algebraic space admitting a dualizing complex and $\phi$ be a Thomason filtration on $X$. 
    Then $\mathcal{U}_\phi$ restricts to $D^b_{\operatorname{coh}}(X)$ if, and only if, $\phi$ satisfies the weak Cousin condition.
\end{theorem}

For the first implication, we show that restrictability ascends to an \'{e}tale presentation by an affine scheme. 
The affine classification of \cite{Takahashi:2023} implies the pulled-back Thomason filtration is weak Cousin, and we show this condition descents to $|X|$. 
Conversely, we showed that truncation triangles behave nicely under derived pullback along \'{e}tale presentations, which allows to descend boundedness and pseudocoherence of truncations. 

To deduce \Cref{thm:dualizing_complex_classification}, we had to also for any $\mathcal{S}\subseteq D^b_{\operatorname{coh}}(X)$, $\overline{\langle \mathcal{S} \rangle}^{(-\infty,0]}_{\otimes}$ is compactly generated.
See \Cref{prop:aisles_are_compactly_generated}.
This extends \cite[Theorem 3.10]{AlonsoTarrio/JeremiasLopez/Saorin:2010} to the setting of algebraic spaces. 
However, this is crucial in order for us to leverage the classification on $D_{\operatorname{qc}}(X)$ via \cite{Hrbek/Lank/Pizzirani:2025} to study the tensor $t$-structures on $D^b_{\operatorname{coh}}$. 

\begin{remark}
    \Cref{thm:restrict_iff_weak_cousin} is stated only for the absolute case $D^b_{\operatorname{coh}}(X)$.
    A corresponding result for the relative case with $D^b_{\operatorname{coh},Z}(X)$ require a support-theoretic upgrade of \Cref{prop:etale_verdier_localization}.
    In fact, for schemes, a relative statement appears in \cite{Clark/Lank/ManaliRahul/Parker:2024} but our methods remain different, nor rely on loc.\ cit.
    This appears difficult because support needs to be suitably controlled along lifts in the factorization from Zariski's Main Theorem.

    An alternative approach can be to apply the induction principle \cite[\href{https://stacks.math.columbia.edu/tag/08GL}{Tag 08GL}]{stacks-project}.
    Although, the issue that is faced occurs when immediate generalizations occur with points in the open subspace and its complement. 
    It is not clear how to control the immediate generalizations for this special case.
\end{remark}

\begin{ack}
    Lank was supported under the ERC Advanced Grant 101095900-TriCatApp, and thanks the Max Planck Institute for Mathematics and the Institute of Mathematics of the Czech Academy of Sciences for their hospitality. The author greatly appreciates discussions with Timothy De Deyn, El\'{i}as Guisado Villalgordo, Michal Hrbek, Giovanna Le Gros, Kabeer Manali Rahul, Sergio Pavon, and Fei Peng.
\end{ack}

\section{Preliminaries}
\label{sec:preliminaries}

Let $\mathcal{T}$ be a triangulated category. 
By `strictly full' subcategory, we mean a full subcategory closed under isomorphisms. 

\subsection{Conventions/notation}

We follow \cite{stacks-project} for algebraic spaces and their derived categories. 
See \cite[\S 2.3]{GuisadoVillaalgordo/Lank:2026} for a summary.
Let $X$ be a quasi-compact quasi-separated algebraic space. 
$\operatorname{Mod}(X)$ is the Grothendieck abelian category of sheaves of $\mathcal{O}_X$-modules on the small \'{e}tale site $X_{\textrm{\'{e}tale}}$ of $X$ \cite[\href{https://stacks.math.columbia.edu/tag/03LT}{Tag 03LT}]{stacks-project}. 
$\operatorname{Qcoh}(X)$ is the full subcategory of $\operatorname{Mod}(X)$ consisting of quasi-coherent sheaves. 
$D(X) := D(\operatorname{Mod}(X))$ is the derived category of $\operatorname{Mod}(X)$. 
$D_{\operatorname{qc}}(X)$ is the full subcategory of $D(X)$ consisting of complexes with quasi-coherent cohomology sheaves. 
$\operatorname{Perf}(X)$ is the full subcategory of perfect complexes in $D_{\operatorname{qc}}(X)$. 
If $X$ is Noetherian, then $\operatorname{coh}(X)$ is the full subcategory of $\operatorname{Mod}(X)$ consisting of coherent sheaves and $D^b_{\operatorname{coh}}(X)$ denotes the full subcategory of $D(X)$ consisting of bounded pseudocoherent complexes.

\subsection{\texorpdfstring{$t$}{t}-structures}
\label{sec:prelim_$t$-structures}

A \textbf{$t$-structure} on $\mathcal{T}$ is a pair of strictly full subcategories $\tau = (\mathcal{T}^{\leq 0}, \mathcal{T}^{\geq 0})$ of $\mathcal{T}$ such that:
\begin{itemize}
    \item $\operatorname{Hom}(A,B) = 0$ for all $A \in \mathcal{T}^{\leq 0}$ and $B \in \mathcal{T}^{\geq 0}[-1]$,
    \item $\mathcal{T}^{\leq 0}[1] \subseteq \mathcal{T}^{\leq 0}$ and $\mathcal{T}^{\geq 0}[-1] \subseteq \mathcal{T}^{\geq 0}$,
    \item for every $E \in \mathcal{T}$, there exists a distinguished triangle
    \begin{displaymath}
        \tau^{\leq 0} E \to E \to \tau^{\geq 1} E \to (\tau^{\leq 0} E)[1]
    \end{displaymath}
    with $\tau^{\leq 0} E \in \mathcal{T}^{\leq 0}$ and $\tau^{\geq 1} E \in \mathcal{T}^{\geq 0}[-1]$.
\end{itemize}
See \cite{Keller/Vossieck:1988,Beilinson/Berstein/Deligne/Gabber:2018} for details. 
The distinguished triangle above is unique up to unique isomorphism. 
It is called the \textbf{truncation triangle} of $E$ with respect to $\tau$. 
For any $n \in \mathbb{Z}$, $(\mathcal{T}^{\leq n}, \mathcal{T}^{\geq n})$ is a $t$-structure on $\mathcal{T}$, where $\mathcal{T}^{\leq n} := \mathcal{T}^{\leq 0}[-n]$ and $\mathcal{T}^{\geq n} := \mathcal{T}^{\geq 0}[-n]$. 
Consider an exact functor $f \colon \mathcal{T}_1 \to \mathcal{T}_2$ between triangulated categories. 
Suppose each $\mathcal{T}_i$ is equipped with $t$-structures $(\mathcal{T}_i^{\leq 0}, \mathcal{T}_i^{\geq 0})$. We say that $f$ is \textbf{right $t$-exact} if $f(\mathcal{T}_1^{\leq 0}) \subseteq \mathcal{T}_2^{\leq 0}$,
and \textbf{left $t$-exact} if 
$f(\mathcal{T}_1^{\geq 0}) \subseteq \mathcal{T}_2^{\geq 0}$. 
If both conditions hold, then $f$ is \textbf{$t$-exact}. 

\subsection{(Pre)aisles}
\label{sec:prelim_$t$-structures_pre_aisle}

A strictly full subcategory $\mathcal{A} \subseteq \mathcal{T}$ is a \textbf{preaisle} if $\mathcal{A}$ is closed under positive shifts and extensions. 
Additionally, it is an \textbf{aisle} if the inclusion $\mathcal{A} \to \mathcal{T}$ admits a right adjoint.
In fact, a subcategory $\mathcal{A}\subseteq\mathcal{T}$ is an aisle if, and only if, the pair $(\mathcal{A},\mathcal{A}^\bot[1])$ is a $t$-structure. 
Here,
\begin{displaymath}
    \mathcal{A}^\perp := \{ T \in \mathcal{T} : \forall A \in \mathcal{A}, \operatorname{Hom}(A,T) = 0  \}.
\end{displaymath}
There is similar notation for `left orthogonals', 
\begin{displaymath}
    {}^\perp \mathcal{A} := \{ T \in \mathcal{T} : \forall A \in \mathcal{A}, \operatorname{Hom}(T,A) = 0  \}.
\end{displaymath}
Respectively, if $(\mathcal{T}^{\leq 0}, \mathcal{T}^{\geq 0})$ is a $t$-structure, $\mathcal{T}^{\leq 0}$ and $\mathcal{T}^{\geq 0}$ are called the \textbf{aisle} and \textbf{coaisle} of the $t$-structure. 
By \cite[Lemma 1.4]{AlonsoTarrio/Lopez/Salorio:2003}, aisles are closed under direct summands of coproducts when they exist in the ambient triangulated category. A preaisle $\mathcal{A}$ is called \textbf{total} if $\mathcal{A} = {}^\perp (\mathcal{A}^\perp)$. Given $\mathcal{C}\subseteq \mathcal{T}$, we write $\mathcal{C}(-\infty,0]$ for the stricly full subcategory of objects of the form $C[n]$ where $n\geq 0$ and $C\in \mathcal{C}$. 

\subsection{`Big' (pre)aisles}
\label{sec:prelim_$t$-structures_big}

Assume that $\mathcal{T}$ admits small coproducts. 
A preaisle $\mathcal{A} \subseteq \mathcal{T}$ is called \textbf{cocomplete} if it is closed under all coproducts in $\mathcal{T}$. 
Given $\mathcal{S}\subseteq \mathcal{T}$, $\overline{\langle \mathcal{S} \rangle}^{(-\infty, 0]}$ is defined to be the smallest cocomplete preaisle containing $\mathcal{S}$. 
If $\mathcal{T}$ is well generated (e.g.\ if it is compactly generated), $\overline{\langle \mathcal{S} \rangle}^{(-\infty, 0]}$ is an aisle whenever $\mathcal{S}$ is essentially small. 
See \cite[Theorem 2.3]{Neeman:2021}. 
An aisle $U$ on $\mathcal{T}$ is \textbf{compactly generated} if there exists a collection of compact objects $\mathcal{P} \subseteq \mathcal{T}^c$ satisfying $\overline{\langle \mathcal{P} \rangle}^{(-\infty, 0]} = U$. 
Hence, we say a $t$-structure is \textbf{compactly generated} if its aisle is as such.  

\begin{example}
    \label{ex:standard_aisle_cpt_gen_for_schemes}
    Let $X$ be a quasi-compact quasi-separated algebraic space. 
    By \cite[Proposition 1.1]{Hrbek/Lank/Pizzirani:2025}', the standard aisle $D^{\leq 0}_{\operatorname{qc}}(X)$ is compactly generated by $\operatorname{Perf}(X) \cap D^{\leq 0}_{\operatorname{qc}}(X)$. 
    See also \cite[Proposition 2.3]{Herbera/Hrbek/LeGros:2025} for the case of Noetherian schemes.
\end{example}

\subsection{Tensor variants}
\label{sec:prelim_$t$-structures_tensor}

Assume that $\mathcal{T}$ is a tensor triangulated category with tensor $\otimes$ and unit $1$. 
Choose a preaisle $\mathcal{P}^{\leq 0} \subseteq \mathcal{T}$ satisfying $\mathcal{P}^{\leq 0} \otimes \mathcal{P}^{\leq 0} \subseteq \mathcal{P}^{\leq 0}$ and $1 \in \mathcal{P}^{\leq 0}$. 
A \textbf{tensor (pre)aisle} (or \textbf{$\otimes$-(pre)aisle}), with respect to $\mathcal{P}^{\leq 0}$, is a (pre)aisle $U \subseteq \mathcal{T}$ which is closed under tensoring by $\mathcal{P}^{\leq 0}$, i.e.\ $\mathcal{P}^{\leq 0} \otimes U \subseteq U$. 
If clear from context, we omit `with respect to $\mathcal{P}^{\leq 0}$'. 
We say a $t$-structure on $\mathcal{T}$ is \textbf{tensor} if its aisle is such. 

If $\mathcal{T}$ admits small coproducts and $\mathcal{C} \subseteq \mathcal{T}$, $\overline{ \langle \mathcal{C} \rangle}^{(-\infty,0]}_{\otimes}$ is defined to be the smallest cocomplete $\otimes$-preaisle containing $\mathcal{C}$.

\begin{remark}
    \label{rmk:HLP_HLLP_proxy_aisle}
    Let $X$ be a quasi-compact quasi-separated algebraic space. 
    Consider any set $\mathcal{C}\subseteq D_{\operatorname{qc}}(X)$. 
    By \Cref{ex:standard_aisle_cpt_gen_for_schemes}, $D^{\leq 0}_{\operatorname{qc}}(X)$ is compactly generated by $D^{\leq 0}_{\operatorname{qc}}(X) \cap \operatorname{Perf}(X)$. 
    Then, from \cite[Lemmas 2.6 \& 2.8]{Hrbek/Lank/LeGros/Pavon:2026}, it follows the tensor aisle is given by
    \begin{displaymath}
        \overline{\langle \mathcal{C} \rangle}^{(-\infty,0]}_{\otimes} = \overline{\langle (D^{\leq 0}_{\operatorname{qc}}(X) \cap \operatorname{Perf}(X)) \otimes^{\mathbf{L}} \mathcal{C} \rangle}^{(-\infty,0]}.
    \end{displaymath}
\end{remark}

While the following uses results proved later on, we record it here.

\begin{lemma}
    \label{lem:dbcoh_tensor_via_perfect_tensor}
    Let $X$ be a Noetherian algebraic space.
    Then an aisle $\mathcal{A}$ on $D^b_{\operatorname{coh}}(X)$ being tensor is equivalent to it being closed under the tensor action by $\operatorname{Perf}(X)\cap D^{\leq 0}_{\operatorname{qc}}(X)$. 
\end{lemma}

\begin{proof}
    It is easy to see that 
    \begin{displaymath}
        \operatorname{Perf}(X)\cap D^{\leq 0}_{\operatorname{qc}}(X)\subseteq D^{\leq 0}_{\operatorname{coh}}(X).
    \end{displaymath}
    If $\mathcal{A}$ is a tensor aisle on $D^b_{\operatorname{coh}}(X)$, then
    \begin{displaymath}
        (\operatorname{Perf}(X)\cap D^{\leq 0}_{\operatorname{qc}}(X) ) \otimes^{\mathbf{L}} \mathcal{A} \subseteq \mathcal{A}.
    \end{displaymath}
    We prove the converse. 
    Let $E\in D^{\leq 0}_{\operatorname{coh}}(X) \otimes^{\mathbf{L}} \mathcal{A}$ where $\mathcal{A}$ is an aisle on $D^b_{\operatorname{coh}}(X)$ which is closed under the tensor action by $\operatorname{Perf}(X) \cap D^{\leq 0}_{\operatorname{qc}}(X)$. 
    Consider the truncation triangle of $E$ with respect to $\mathcal{A}$ on $D^b_{\operatorname{coh}}(X)$,
    \begin{displaymath}
        \tau^{\leq 0} E \to E \to \tau^{\geq 1} E \to \tau^{\leq 0} E.
    \end{displaymath}
    By \Cref{prop:aisles_are_compactly_generated} and \Cref{rmk:HLP_HLLP_proxy_aisle},
    \begin{displaymath}
        \overline{\langle \mathcal{A}^\prime \rangle}^{(-\infty,0]}_{\otimes} = \overline{\langle (D^{\leq 0}_{\operatorname{qc}}(X) \cap \operatorname{Perf}(X)) \otimes^{\mathbf{L}} \mathcal{A}^\prime \rangle}^{(-\infty,0]}.
    \end{displaymath}
    We claim the distinguished triangle above corresponds to the truncation triangle for $E$ with respect to $\overline{\langle \mathcal{A} \rangle}^{(-\infty,0]}_{\otimes}$ on $D_{\operatorname{qc}}(X)$. 
    Note that $\tau^{\geq 1} E \in \mathcal{A}^\perp \cap D^b_{\operatorname{coh}}(X)$.
    By \Cref{lem:aisle_via_double_orthogonal}, $\tau^{\geq 1} E\in (\overline{\langle \mathcal{A} \rangle}^{(-\infty,0]}_{\otimes})^\perp$. 
    However, $E\in D^{\leq 0}_{\operatorname{coh}}(X) \otimes^{\mathbf{L}} \mathcal{A}$ and so $E\in \overline{\langle \mathcal{A} \rangle}^{(-\infty,0]}_{\otimes}$. 
    Thus, the morphism $\tau^{\leq 0} E \to E$ is an isomorphism, which implies $E\in \mathcal{A}$.
\end{proof}

\subsection{Thomason filtrations}
\label{sec:prelim_Thomason}

We recall Thomason filtrations on a Noetherian algebraic space $X$. 
Let $\operatorname{Spcl}(|X|)$ be the collection of specialization closed subsets of $|X|$. 
A function $\phi \colon \mathbb{Z} \to \operatorname{Spcl}(|X|)$ is called a \textbf{Thomason filtration} on $X$ if $\phi(n+1)\subseteq \phi (n)$ for each $n\in \mathbb{Z}$. 

Consider a morphism $f\colon Y\to X$ of Noetherian algebraic spaces. 
Fix a Thomason filtration $\phi$ on $X$. 
Define $|f|^{-1}\phi$ to be the function $\mathbb{Z} \to \operatorname{Spcl}(|Y|)$ given by the rule $n\mapsto |f|^{-1}(\phi(n))$. 
It is straightforward to check that $|f|^{-1}(\phi(n))$ is specialization closed because $f$ induces a continuous function $|Y|\to |X|$.

A generalization $p\rightsquigarrow q$ (i.e.\ $q\in \overline{\{ p\}}$) is \textbf{immediate} if $p\not=q$ and there does not exist a point $s$ in $|X|$ such that $p\rightsquigarrow s \rightsquigarrow q$ where $p\not=s$.
We say that a Thomason filtration $\phi \colon \mathbb{Z} \to \operatorname{Spcl}(|X|)$ is said to satisfy the \textbf{weak Cousin condition} if for every integer $n$, the set $\phi (n-1)$ contains all direct generalizations of each point in $\phi (n)$ (i.e.\ for each direct generalization  $p \rightsquigarrow q$ with $q$ in $\phi (n)$, one has $p$ is in $\phi (n-1)$).

By \cite[Theorem 1.4]{Hrbek/Lank/Pizzirani:2025}, there exists a one-to-one correspondence:
\begin{displaymath}
    \begin{aligned}
        \{ \otimes\textrm{-aisle on } & D_{\operatorname{qc}}(X) \textrm{ generated by } \mathcal{P} \subseteq \operatorname{Perf}(X) \} 
        \\&\iff \{ \textrm{Thomason filtrations on } X\}.
    \end{aligned}
\end{displaymath}
In particular, for any $\otimes$-aisle on $D_{\operatorname{qc}}(X)$ generated by $\mathcal{P} \subseteq \operatorname{Perf}(X)$, we assign to it the Thomason filtration given by the rule (which is independent of the choice for $\mathcal{P}$)
\begin{displaymath}
    n\mapsto \bigcup_{\substack{j\geq n\\ P\in \mathcal{P}}} \operatorname{Supp}(\mathcal{H}^j (P)).
\end{displaymath}
Given a Thomason filtration $\phi$ on $X$, we denote the \textbf{corresponding} compactly generated tensor $t$-structure on $D_{\operatorname{qc}}(X)$ by $(\mathcal{U}_\phi,\mathcal{V}_\phi)$, and by $\tau^{\leq 0}_\phi$ and $\tau^{\geq 1}_{\phi}$ its truncations. 

\section{Exactness}
\label{sec:exactness}

\begin{lemma}
    \label{lem:aisle_via_double_orthogonal}
    Let $\mathcal{T}$ be a well-generated triangulated category. 
    Consider any set $\mathcal{C}\subseteq \mathcal{T}$. 
    If $\mathcal{C}[1]\subseteq \mathcal{C}$, then $\overline{\langle \mathcal{C} \rangle}^{(-\infty,0]} = {}^\perp (\mathcal{C}^\perp)$. 
    In particular, the coaisle is given by $\mathcal{C}^\perp[1]$.
\end{lemma}

\begin{proof}
    By \cite[Theorem 2.3]{Neeman:2021}, $\overline{\langle \mathcal{C} \rangle}^{(-\infty,0]}$ is the smallest cocomplete aisle of $\mathcal{T}$ containing $\mathcal{C}$. 
    Since ${}^\perp (\mathcal{C}^\perp)$ is a cocomplete total preaisle containing $\mathcal{C}$, we know that $\overline{\langle \mathcal{C} \rangle}^{(-\infty,0]} \subseteq {}^\perp (\mathcal{C}^\perp)$. 
    The reverse inclusion follows from the fact that $\overline{\langle \mathcal{C} \rangle}^{(-\infty,0]}$ is a total preaisle containing $\mathcal{C}$. 
    It is straightforward to check the claim for the coaisle.
\end{proof}

\begin{lemma}
    \label{lem:pushforward_sends_shift_of_aisle_to_aisle}
    Let $f\colon Y \to X$ be a morphism of quasi-compact quasi-separated algebraic spaces. 
    Then there exists an $n\geq 0$ such that $\mathbf{R}f_\ast D^{\leq -n}_{\operatorname{qc}}(Y) \subseteq D^{\leq 0}_{\operatorname{qc}}(X)$.
\end{lemma}

\begin{proof}
    This follows from the proof of \cite[\href{https://stacks.math.columbia.edu/tag/08FA}{Tag 08FA}]{stacks-project}.
    With notation in loc.\ cit., $N$ is chosen from \cite[\href{https://stacks.math.columbia.edu/tag/073G}{Tag 073G}]{stacks-project}.
    The latter result reduces to affine source and target with parameters $d_p$ and $p$ coming from \cite[\href{https://stacks.math.columbia.edu/tag/072C}{Tag 072C}]{stacks-project}. 
    In particular, these values are nonnegative, and so $N\geq 0$. 
    Consequently, there exists $N\geq 0$ such that $\mathbf{R}f_\ast D^{\leq 0}_{\operatorname{qc}}(Y) \subseteq D^{\leq N}_{\operatorname{qc}}(X)$. 
    By shifting, it follows that 
    \begin{displaymath}
        \mathbf{R}f_\ast D^{\leq -N}_{\operatorname{qc}}(Y) 
        \subseteq \mathbf{R}f_\ast D^{\leq 0}_{\operatorname{qc}}(Y)[N]
        \subseteq D^{\leq 0}_{\operatorname{qc}}(X).
    \end{displaymath}
    This completes the proof.
\end{proof}

\begin{proposition}
    \label{prop:pushforward_tensor_aisle_twisted_to_aisle}
    Let $f\colon Y\to X$ be a morphism of Noetherian algebraic spaces.
    Let $\mathcal{C}\subseteq D_{\operatorname{qc}}(X)$ be a set such that $\mathcal{C}[1]\subseteq \mathcal{C}$. 
    Then there exists $N\geq 0$ such that 
    \begin{displaymath}
        \mathbf{R}f_\ast \overline{\langle \mathbf{L}f^\ast \mathcal{C}[N] \rangle}^{(-\infty,0]}_{\otimes} \subseteq \overline{\langle \mathcal{C} \rangle}^{(-\infty,0]}_{\otimes}.
    \end{displaymath}
\end{proposition}

\begin{proof}
    By \Cref{lem:pushforward_sends_shift_of_aisle_to_aisle}, there exists an $N\geq 0$ such that $\mathbf{R}f_\ast D^{\leq -N}_{\operatorname{qc}}(Y) \subseteq D^{\leq 0}_{\operatorname{qc}}(X)$. 
    Applying \Cref{rmk:HLP_HLLP_proxy_aisle}, we obtain
    \begin{displaymath}
        \overline{\langle \mathbf{L}f^\ast \mathcal{C}[N] \rangle}^{(-\infty,0]}_{\otimes}
        = \overline{\langle (D^{\leq 0}_{\operatorname{qc}}(Y) \cap \operatorname{Perf}(Y)) \otimes^{\mathbf{L}} \mathbf{L}f^\ast \mathcal{C}[N] \rangle}^{(-\infty,0]}.
    \end{displaymath}
    Similar reasoning yields
    \begin{displaymath}
        \overline{\langle \mathcal{C} \rangle}^{(-\infty,0]}_{\otimes}
        = \overline{\langle (D^{\leq 0}_{\operatorname{qc}}(X) \cap \operatorname{Perf}(X)) \otimes^{\mathbf{L}} \mathcal{C} \rangle}^{(-\infty,0]}.
    \end{displaymath}
    
    For the sake of notation, set $\mathcal{P}:= D^{\leq 0}_{\operatorname{qc}}(X) \cap \operatorname{Perf}(X)$ and $\mathcal{Q}:= D^{\leq 0}_{\operatorname{qc}}(Y) \cap \operatorname{Perf}(Y)$. 
    Observe that $\mathcal{Q} \otimes^{\mathbf{L}} \mathbf{L}f^\ast \mathcal{C}[N]$ and $\mathcal{P} \otimes^{\mathbf{L}} \mathcal{C}$ are sets which are closed under nonnegative shifts. 
    By \Cref{lem:aisle_via_double_orthogonal}, it follows that 
    \begin{displaymath}
        \overline{\langle \mathcal{Q} \otimes^{\mathbf{L}} \mathbf{L}f^\ast \mathcal{C}[N] \rangle}^{(-\infty,0]} = {}^\perp ( ( \mathcal{Q} \otimes^{\mathbf{L}} \mathbf{L}f^\ast \mathcal{C}[N] )^\perp )
    \end{displaymath}
    and 
    \begin{displaymath}
        \overline{\langle \mathcal{P} \otimes^{\mathbf{L}} \mathcal{C} \rangle}^{(-\infty,0]} = {}^\perp ( ( \mathcal{P} \otimes^{\mathbf{L}} \mathcal{C} )^\perp ).
    \end{displaymath}
    Hence, the coaisles are respectively given by $( \mathcal{Q} \otimes^{\mathbf{L}} \mathbf{L}f^\ast \mathcal{C}[N] )^\perp$ and $( \mathcal{P} \otimes^{\mathbf{L}} \mathcal{C} )^\perp$. 

    Denote by $f^\times$ the right adjoint of $\mathbf{R}f_\ast$ on $D_{\operatorname{qc}}$ \cite[\href{https://stacks.math.columbia.edu/tag/0E55}{Tag 0E55}]{stacks-project}. 
    We claim that $f^\times ( \mathcal{P} \otimes^{\mathbf{L}} \mathcal{C} )^\perp \subseteq ( \mathcal{Q} \otimes^{\mathbf{L}} \mathbf{L}f^\ast \mathcal{C}[N] )^\perp$. 
    Let $E\in ( \mathcal{P} \otimes^{\mathbf{L}} \mathcal{C} )^\perp$. 
    Choose $A\in \mathcal{Q} \otimes^{\mathbf{L}} \mathbf{L}f^\ast \mathcal{C}[N]$. 
    There exist $Q\in \mathcal{Q}$ and $C\in \mathcal{C}[N]$ such that $A\cong Q \otimes^{\mathbf{L}} \mathbf{L} f^\ast C[N]$. 
    By adjunction and projection formula, 
    \begin{displaymath}
        \begin{aligned}
            \operatorname{Hom}(A,f^\times E)
            &\cong \operatorname{Hom}(Q \otimes^{\mathbf{L}} \mathbf{L} f^\ast C[N], f^\times E)
            \\&\cong \operatorname{Hom}(\mathbf{R}f_\ast (Q \otimes^{\mathbf{L}} \mathbf{L} f^\ast C[N]), E)
            \\&\cong \operatorname{Hom}(\mathbf{R}f_\ast Q[N] \otimes^{\mathbf{L}} C, E).
        \end{aligned}
    \end{displaymath}
    By the choice of $N$,
    \begin{displaymath}
        \mathbf{R}f_\ast Q[N]\in \mathbf{R}f_\ast D^{\leq -N}_{\operatorname{qc}}(Y) \subseteq D^{\leq 0}_{\operatorname{qc}}(X),
    \end{displaymath}
    and so,
    \begin{displaymath}
        \mathbf{R}f_\ast Q[N] \otimes^{\mathbf{L}} C \in \overline{\langle \mathcal{C} \rangle}^{(-\infty,0]}_{\otimes} = {}^\perp ( ( \mathcal{P} \otimes^{\mathbf{L}} \mathcal{C} )^\perp ).
    \end{displaymath}
    Since $E \in ( \mathcal{P} \otimes^{\mathbf{L}} \mathcal{C} )^\perp$, it follows that
    \begin{displaymath}
        \operatorname{Hom}(\mathbf{R}f_\ast Q[N] \otimes^{\mathbf{L}} C, E) \cong 0.
    \end{displaymath}
    Thus, $f^\times E\in  ( \mathcal{Q} \otimes^{\mathbf{L}} \mathbf{L}f^\ast \mathcal{C}[N] )^\perp$.

    Now, we finish the proof. 
    Choose $E\in \overline{\langle \mathcal{Q} \otimes^{\mathbf{L}} \mathbf{L}f^\ast \mathcal{C}[N] \rangle}^{(-\infty,0]}$. 
    Let $G\in ( \mathcal{P} \otimes^{\mathbf{L}} \mathcal{C} )^\perp$. 
    By adjunction,
    \begin{displaymath}
        \operatorname{Hom}(\mathbf{R}f_\ast E , G)
        \cong \operatorname{Hom}(E , f^\times G)
        \cong 0.
    \end{displaymath}
    Indeed, the work above shows $f^\times G \in ( \mathcal{Q} \otimes^{\mathbf{L}} \mathbf{L}f^\ast \mathcal{C}[N] )^\perp$ (i.e.\ the coaisle), and $E$ belongs to the aisle. 
    Therefore, we have shown that 
    \begin{displaymath}
        \mathbf{R}f_\ast E 
        \in {}^\perp ( ( \mathcal{P} \otimes^{\mathbf{L}} \mathcal{C} )^\perp )
        = \overline{\langle \mathcal{P} \otimes^{\mathbf{L}} \mathcal{C} \rangle}^{(-\infty,0]}.
    \end{displaymath}
    This completes the proof.
\end{proof}

\begin{placeholder}
    Let $X$ be a Noetherian algebraic space. 
    We say $X$ satisfies \textbf{B.C.C.} if for every set $\mathcal{S}\subseteq D^b_{\operatorname{coh}}(X)$ there exists $n\geq 0$ such that $D^{\leq -n}_{\operatorname{qc},Z}(X) \subseteq \overline{\langle \mathcal{S} \rangle}^{(-\infty,0]}_{\otimes}$ where $Z:= \cup_{E\in \mathcal{S}} \operatorname{Supp}(E)$.
\end{placeholder}

\begin{proposition}
    \label{prop:etale_nbhd}
    Let $X$ be a Noetherian algebraic space. 
    Consider an elementary distinguished square
    \begin{displaymath}
        \begin{tikzcd}
            {U^\prime} & {V} \\
            {U} & {X}
            \arrow["{j^\prime}", from=1-1, to=1-2]
            \arrow["{f^\prime}"', from=1-1, to=2-1]
            \arrow["f", from=1-2, to=2-2]
            \arrow["j"', from=2-1, to=2-2]
        \end{tikzcd}
    \end{displaymath}
    where $j$ is an open immersion and $f$ an \'{e}tale morphism from an affine scheme. 
    If $U$ satisfies B.C.C., then $X$ satisfies B.C.C.
\end{proposition}

\begin{proof}
    Fix a set $\mathcal{S}\subseteq D^b_{\operatorname{coh}}(X)$. 
    Note that $Z:= \cup_{E\in \mathcal{S}} \operatorname{Supp}(E)$ is specialization closed. 
    Applying \Cref{prop:pushforward_tensor_aisle_twisted_to_aisle}, we can find $N\gg 0$ such that 
    \begin{displaymath}
        \mathbf{R}g_\ast \overline{\langle \mathbf{L}g^\ast \mathcal{S}[N] (-\infty,0] \rangle}^{(-\infty,0]}_{\otimes} 
        \subseteq \overline{\langle \mathcal{S}(-\infty,0] \rangle}^{(-\infty,0]}_{\otimes}
    \end{displaymath}
    for each $g\in \{f,j,f\circ j^\prime \}$. 

    By \cite[\href{https://stacks.math.columbia.edu/tag/071Q}{Tags 071Q}, \href{https://stacks.math.columbia.edu/tag/08H8}{08H8}, \& \href{https://stacks.math.columbia.edu/tag/07TY}{07TY}]{stacks-project}, we can transfer \cite[Lemma 3.3]{Hrbek/Lank/LeGros/Pavon:2026} on the small Zariski site of a scheme to the small \'{e}tale site. 
    Hence, $U^\prime$ and $V$ satisfy B.C.C. 
    As each morphism in the elementary distinguished square is flat, the derived pullback of any object from $\mathcal{S}$ remains a bounded pseudocoherent complex. 
    Recall that the hypothesis says $U$ satisfies B.C.C. 
    Thus, there exist $n_1,n_2,n_3\geq 0$ such that 
    \begin{itemize}
        \item $D^{\leq -n_1}_{\operatorname{qc},Z \cap |\mathcal{U|}}(U) 
        \subseteq \overline{\langle \mathbf{L}j^\ast \mathcal{S}[N] (-\infty,0] \rangle}^{(-\infty,0]}_{\otimes}$
        \item $D^{\leq -n_2}_{\operatorname{qc},f^{-1}}(Z)(V) 
        \subseteq \overline{\langle \mathbf{L}f^\ast \mathcal{S}[N] (-\infty,0]\rangle}^{(-\infty,0]}_{\otimes}$
        \item $D^{\leq -n_3}_{\operatorname{qc},(f\circ j^\prime)^{-1}(Z)}(U^\prime) 
        \subseteq \overline{\langle \mathbf{L}(f\circ j^\prime)^\ast \mathcal{S}[N] (-\infty,0] \rangle}^{(-\infty,0]}_{\otimes}$.
    \end{itemize}
    Set $n=\max\{n_1,n_2,n_3\}$. 

    By \cite[Theorem 1.4]{Hrbek/Lank/Pizzirani:2025}, $D^{\leq -n}_{\operatorname{qc},Z}(X)$ is compactly generated by $P\in \operatorname{Perf}(X)$ which belong to $D^{\leq -n}_{\operatorname{qc}}(X)$ with support in $Z$. 
    Indeed, $D^{\leq -n}_{\operatorname{qc},Z}(X)$ corresponds to the Thomason filtration given by the rule $N \mapsto Z$ for $N \leq -n$ and $N \mapsto \emptyset$ for $-n < N$.
    It suffices to show that every such $P$ belongs to $\overline{\langle \mathcal{S} \rangle}^{(-\infty,0]}_{\otimes}$. 

    Fix $P\in \operatorname{Perf}(X) \cap D^{\leq -n}_{\operatorname{qc},Z}(X)$.  
    Set $W=\operatorname{Supp}(P)$. 
    Note that $W\subseteq Z$. 
    By \cite[\href{https://stacks.math.columbia.edu/tag/08GW}{Tags 08GW}]{stacks-project}, there exists a distinguished triangle 
    \begin{displaymath}
        P \to \mathbf{R}j_\ast \mathbf{L}j^\ast P \oplus \mathbf{R}f_\ast \mathbf{L}f^\ast P \to \mathbf{R}(j \circ f^\prime)_\ast \mathbf{L}(j \circ f^\prime)^\ast P \to P[1].
    \end{displaymath}
    Applying \cite[Lemma A.2]{GuisadoVillaalgordo/Lank:2026}, it follows that 
    \begin{displaymath}
        \mathbf{L}j^\ast P \in D^{\leq -n}_{\operatorname{qc},W \cap |\mathcal{U|}}(U) \subseteq D^{\leq -n_1}_{\operatorname{qc},W \cap |\mathcal{U|}}(U) \subseteq \overline{\langle \mathbf{L}j^\ast \mathcal{S}[N] (-\infty,0] \rangle}^{(-\infty,0]}_{\otimes}
    \end{displaymath}
    and 
    \begin{displaymath}
        \mathbf{L}f^\ast P \in D^{\leq -n}_{\operatorname{qc},f^{-1}(W)}(V) \subseteq D^{\leq -n_2}_{\operatorname{qc},f^{-1}(W)}(V) \subseteq  \overline{\langle \mathbf{L}f^\ast \mathcal{S}[N] (-\infty,0]\rangle}^{(-\infty,0]}_{\otimes},
    \end{displaymath}
    as well as,
    \begin{displaymath}
        \begin{aligned}
            \mathbf{L}(f\circ j^\prime)^\ast P 
            &\in D^{\leq -n}_{\operatorname{qc},(f\circ j^\prime)^{-1}(W)}(U^\prime) 
            \\&\subseteq D^{\leq -n_3}_{\operatorname{qc},(f\circ j^\prime)^{-1}(W)}(U^\prime) \subseteq \overline{\langle \mathbf{L}(f\circ j^\prime)^\ast \mathcal{S}[N] (-\infty,0] \rangle}^{(-\infty,0]}_{\otimes}.
        \end{aligned}
    \end{displaymath}
    It follows, from the distinguished triangle above, that $P[1]$ belongs to $\overline{\langle \mathcal{S}(-\infty,0] \rangle}^{(-\infty,0]}_{\otimes}$. Indeed, $P[1]$ is an extension of objects from $\overline{\langle \mathcal{S}(-\infty,0] \rangle}^{(-\infty,0]}_{\otimes}$, and aisles are closed under extensions.
    Since $n$ is independent of $P$, this completes the proof.
    Indeed, the aisle compactly generated by the $P[1]$ is $D^{\leq -n}_{\operatorname{qc},Z}(X)[1]= D^{\leq -n - 1}_{\operatorname{qc},Z}(X)$. 
\end{proof}

\begin{theorem}
    \label{thm:bounded_below}
    Let $X$ be a Noetherian algebraic space. 
    Then $X$ satisfies B.C.C. 
\end{theorem}

\begin{proof}
    We apply \cite[\href{https://stacks.math.columbia.edu/tag/09IT}{Tag 09IT}]{stacks-project}. 
    By \cite[\href{https://stacks.math.columbia.edu/tag/06NH}{Tag 06NH}]{stacks-project}, there exists a dense open subspace $U$ of $X$ which is a scheme. 
    Applying \cite[\href{https://stacks.math.columbia.edu/tag/071Q}{Tags 071Q}, \href{https://stacks.math.columbia.edu/tag/08H8}{08H8}, \& \href{https://stacks.math.columbia.edu/tag/07TY}{07TY}]{stacks-project}, we can transfer \cite[Lemma 3.3]{Hrbek/Lank/LeGros/Pavon:2026} on the small Zariski site of $U$ to the small \'{e}tale site. 
    Hence, $U$ satisfies B.C.C. 
    Thus, the remaining condition to check in \cite[\href{https://stacks.math.columbia.edu/tag/09IT}{Tag 09IT}]{stacks-project} follows from \Cref{prop:etale_nbhd}.
\end{proof}

\begin{lemma}
    \label{lem:bounded_below_t_exactness_flat_pullback}
    Let $X$ be a Noetherian algebraic space, $f\colon U \to X$ be a finitely presented flat morphism from an affine scheme, and $\mathcal{S}\subseteq D^b_{\operatorname{coh}}(X)$. 
    For any $E\in D^+_{\operatorname{qc}}(X)\cap (\overline{\langle \mathcal{S} \rangle}^{(-\infty,0]}_{\otimes})^\perp [1]$, we have $\mathbf{L}f^\ast E \in (\overline{\langle \mathbf{L}f^\ast \mathcal{S} \rangle}^{(-\infty,0]})^\perp [1]$.
\end{lemma}

\begin{proof}
    Assume the contrary. 
    Hence, there exists $E\in D^+_{\operatorname{qc}}(X)\cap (\overline{\langle \mathcal{S} \rangle}^{(-\infty,0]}_{\otimes})^\perp [1]$ such that 
    \begin{displaymath}
        \mathbf{L}f^\ast E \not\in (\overline{\langle \mathbf{L}f^\ast \mathcal{S} \rangle}^{(-\infty,0]}_{\otimes})^\perp [1].
    \end{displaymath}
    Observe that every aisle on $D_{\operatorname{qc}}(U)$ is tensor compatible because $U$ is affine, which implies $\mathbf{L} f^\ast E [-1]\not\in (\overline{\langle \mathbf{L}f^\ast \mathcal{S} \rangle}^{(-\infty,0]})^\perp$. 
    Moreover, using \cite[Lemma 3.1]{AlonsoTarrio/Lopez/Salorio:2003}, we can find an $A\in \mathcal{S}$ and $j\geq 0$ such that $\operatorname{Hom}(\mathbf{L}f^\ast A[j] , \mathbf{L}f^\ast E[-1])\not=0$. 
    Now, we use adjunction to see that,
    \begin{displaymath}
        \begin{aligned}
            \operatorname{Hom} & (\mathbf{L}f^\ast A[j] , \mathbf{L}f^\ast E[-1]) 
            \\&\cong \operatorname{Hom}(\mathcal{O}_U, \operatorname{\mathbf{R}\mathcal{H}\! \mathit{om}}(\mathbf{L}f^\ast A[j], \mathbf{L}f^\ast E[-1]) ) && (\textrm{\cite[\href{https://stacks.math.columbia.edu/tag/08J7}{Tag 08J7}]{stacks-project}})
            \\&\cong \operatorname{Hom}(\mathcal{O}_U, \mathbf{L}f^\ast \operatorname{\mathbf{R}\mathcal{H}\! \mathit{om}}(A[j], E[-1]) ) 
            && (\textrm{\cite[Lemma 2.11]{GuisadoVillaalgordo/Lank:2026}}).
        \end{aligned}
    \end{displaymath}
    Note that we are using the following identification above:
    \begin{displaymath}
        \operatorname{\mathbf{R}\mathcal{H}\! \mathit{om}}(A[j], E[-1]) \cong \operatorname{\mathbb{R}\mathcal{H}\! \mathit{om}}(A[j], E[-1]).
    \end{displaymath}
    This follows from \cite[Lemma 2.7]{GuisadoVillaalgordo/Lank:2026} and  \cite[\href{https://stacks.math.columbia.edu/tag/0A8A}{Tag 0A8A}]{stacks-project}. 
    Since $\mathcal{O}_U$ compactly generates $D^{\leq 0}_{\operatorname{qc}}(U)$, this implies that
    \begin{displaymath}
        \mathbf{L}f^\ast \operatorname{\mathbf{R}\mathcal{H}\! \mathit{om}}(A[j], E[-1]) \not\in D^{\geq 1}_{\operatorname{qc}}(U).
    \end{displaymath}
   Moreover, as $\mathbf{L}f^\ast$ is $t$-exact with respect to the standard $t$-structures, we obtain
    \begin{displaymath}
        \operatorname{\mathbf{R}\mathcal{H}\! \mathit{om}}(A[j], E[-1]) \not\in D^{\geq 1}_{\operatorname{qc}}(X).
    \end{displaymath}
    Then we can find a $B\in D^{\leq 0}_{\operatorname{qc}}(X)$ such that
    \begin{displaymath}
        \begin{aligned}
            0 &\not= 
            \operatorname{Hom}(B , \operatorname{\mathbf{R}\mathcal{H}\! \mathit{om}}(A[j], E[-1]) )
            \\&\cong \operatorname{Hom}(B \otimes^{\mathbf{L}} A[j], E[-1] ) && (\textrm{\cite[\href{https://stacks.math.columbia.edu/tag/08J7}{Tag 08J7}]{stacks-project}}).
        \end{aligned}
    \end{displaymath}
    However, $E[-1]\in (\overline{\langle \mathcal{S} \rangle}^{(-\infty,0]}_{\otimes})^\perp$ and $B\otimes^{\mathbf{L}} A[j]\in \overline{\langle \mathcal{S} \rangle}^{(-\infty,0]}_{\otimes}$, which is a contradiction.
\end{proof}

\begin{proof}
    [Proof of \Cref{prop:t_exactness_flat_pullback}]
    It is straightforward to see that $\mathbf{L}s^\ast \overline{\langle \mathcal{S} \rangle}^{(-\infty,0]}_{\otimes} \subseteq \overline{\langle \mathbf{L}s^\ast \mathcal{S} \rangle}^{(-\infty,0]}_{\otimes}$, which tells us $\mathbf{L}s^\ast$ is right $t$-exact. 
    Applying \Cref{lem:aisle_via_double_orthogonal} and \Cref{rmk:HLP_HLLP_proxy_aisle}, $\overline{\langle \mathcal{S} \rangle}^{(-\infty,0]}_{\otimes} = {}^\perp (\widetilde{\mathcal{S}}^\perp)$ where $\widetilde{\mathcal{S}}$ is given by
    \begin{displaymath}
       (D^{\leq 0}_{\operatorname{qc}}(X) \cap \operatorname{Perf}(X)) \otimes^{\mathbf{L}} \mathcal{S}.
    \end{displaymath}   
    Set
    \begin{displaymath}
        Z:= \bigcup_{S\in \widetilde{\mathcal{S}}} \operatorname{Supp}(S). 
    \end{displaymath}
    Let $B\in (\widetilde{\mathcal{S}})^\perp$. 
    By \cite[Lemma 6.9]{Lank:2026a}, there exists a distinguished triangle
    \begin{displaymath}
        \Gamma_Z (B) \to B \to L_Z (B) \to \Gamma_Z (B)[1]
    \end{displaymath}
    where $\Gamma_Z (B) \in D_{\operatorname{qc},Z}(X)$ and $L_Z (B) \in D_{\operatorname{qc},Z}(X)^\perp$. 
    Since $\overline{\langle \mathcal{S} \rangle}^{(-\infty,0]}_{\otimes} \subseteq D_{\operatorname{qc},Z}(X)$, it follows that $L_Z (B)\in (\widetilde{\mathcal{S}})^\perp$.
    Hence, from the distinguished triangle 
    \begin{displaymath}
        L_Z (B)[-1] \to \Gamma_Z (B) \to B \to L_Z (B),
    \end{displaymath}
    $\Gamma_Z (B)\in (\widetilde{\mathcal{S}})^\perp$ because coaisles are closed under nonpositive shifts and extensions, whereas $B,L_Z (B)\in (\widetilde{\mathcal{S}})^\perp$.
    By \cite[Lemma 6.9]{Lank:2026a}, there exists a distinguished triangle
    \begin{displaymath}
        \Gamma_{|s|^{-1} (Z)} (\mathbf{L}s^\ast B) \to \mathbf{L}s^\ast B \to L_{|s|^{-1} (Z)} (\mathbf{L}s^\ast B) \to \Gamma_{|s|^{-1} (Z)} (\mathbf{L}s^\ast B)[1]
    \end{displaymath}
    where $\Gamma_{|s|^{-1} (Z)} (\mathbf{L}s^\ast B) \in D_{\operatorname{qc},|s|^{-1}(Z)}(U)$ and $L_{|s|^{-1}(Z)} (\mathbf{L}s^\ast B) \in D_{\operatorname{qc},|s|^{-1}(Z)}(U)^\perp$, and it is isomorphic to the distinguished triangle 
    \begin{displaymath}
        \mathbf{L}s^\ast \Gamma_Z (B) \to \mathbf{L}s^\ast B \to \mathbf{L}s^\ast L_Z (B) \to \mathbf{L}s^\ast \Gamma_Z (B)[1].
    \end{displaymath}
    Thus, we obtain that $\mathbf{L}s^\ast L_Z (B)$ belongs to the coaisle of the tensor $t$-structure on $D_{\operatorname{qc}}(U)$ generated by $\mathbf{L}s^\ast \mathcal{S}$.
    We claim that $\Gamma_Z (B)$ is in $D^+_{\operatorname{qc}}(X)$. 
    By \Cref{thm:bounded_below}, there exists $n\geq 0$ such that $D^{\leq -n}_{\operatorname{qc},Z}(X) \subseteq \overline{\langle \mathcal{S} \rangle}^{(-\infty,0]}_{\otimes}$.
    Hence, $(\widetilde{\mathcal{S}})^\perp\cap D_{\operatorname{qc},Z}(X) \subseteq D^{\geq -n+1}_{\operatorname{qc},Z}(X)$.
    Since $\Gamma_Z (B)\in D_{\operatorname{qc},Z}(X)\cap (\widetilde{\mathcal{S}})^\perp$, it follows that $\Gamma_Z (B) \in D^{\geq -n+1}_{\operatorname{qc},Z}(X) \subseteq D^+_{\operatorname{qc}}(X)$. 
    Then \Cref{lem:bounded_below_t_exactness_flat_pullback} implies $\mathbf{L}s^\ast \Gamma_Z (B)$ belongs to the coaisle of the tensor $t$-structure on $D_{\operatorname{qc}}(U)$ generated by $\mathbf{L}s^\ast \mathcal{S}$.
    Thus, $\mathbf{L}s^\ast B$ is an extension of two objects which belong to the coaisle of the tensor $t$-structure on $D_{\operatorname{qc}}(U)$ generated by $\mathbf{L}s^\ast \mathcal{S}$, so this completes the proof.
\end{proof}

\section{Restrictability}
\label{sec:restrictability}

\begin{lemma}
    \label{lem:lifting_descending_immediate_generalizations}
    Let $X$ be a Noetherian algebraic space admitting a dualizing complex. 
    Choose an \'{e}tale presentation $s\colon U \to X$. 
    If $p \rightsquigarrow p^\prime$ is an immediate specialization in $|X|$, then there exists an immediate generalization $u \rightsquigarrow u^\prime$ in $|U|$ such that $s(u) = p$ and $s(u^\prime) = p^\prime$. 
    Moreover, if $u \rightsquigarrow u^\prime$ is an immediate specialization in $|U|$, then $s(u) \rightsquigarrow s(u^\prime)$ is an immediate specialization in $|X|$. 
\end{lemma}

\begin{proof}
    Since $X$ admits a dualizing complex, $|X|$ admits a dimension function \cite[\href{https://stacks.math.columbia.edu/tag/0E53}{Tag 0E53}]{stacks-project}.
    We prove the first claim. 
    Suppose $p \rightsquigarrow p^\prime$ is an immediate specialization in $|X|$. 
    By \cite[\href{https://stacks.math.columbia.edu/tag/03JX}{Tags 03JX}, \href{https://stacks.math.columbia.edu/tag/03JV}{03JV}, \& \href{https://stacks.math.columbia.edu/tag/03K2}{03K2}]{stacks-project}, there exists a generalization $u \rightsquigarrow u^\prime$ in $|U|$ such that $s(u) = p$ and $s(u^\prime) = p^\prime$. 
    Choose $v\in |U|$ such that $u \rightsquigarrow v \rightsquigarrow u^\prime$.
    Then either $s(v)=p$ or $s(v)=p^\prime$ because $p \rightsquigarrow p^\prime$ is an immediate specialization in $|X|$. 
    Appealing to \cite[\href{https://stacks.math.columbia.edu/tag/0H1Q}{Tag 0H1Q}]{stacks-project}, it follows that either $u=v$ or $u^\prime = v$.
    This proves the first claim.  

    Now we check the second claim. 
    Let $u \rightsquigarrow u^\prime$ be an immediate specialization in $|U|$. 
    Assume that $s(u) \rightsquigarrow s(u^\prime)$ is not an immediate specialization in $|X|$. 
    This implies there exists $q\in |X|$ such that $s(u) \rightsquigarrow q \rightsquigarrow s(u^\prime)$ where $s(u)\not=q$ and $s(u^\prime)\not=q$. 
    If necessary, we can impose that $s(u) \rightsquigarrow q$ be an immediate specialization. 
    As $u \rightsquigarrow u^\prime$ is an immediate specialization, we know that $\dim \mathcal{O}_{U,u} + 1 = \dim \mathcal{O}_{U,u^\prime}$. 
    
    Suppose $\overline{u}\colon \operatorname{Spec}(k)\to U$ and $\overline{u}^\prime \colon \operatorname{Spec}(k^\prime) \to U$ are geometric points respectively for $u$ and $u^\prime$. 
    Applying \cite[\href{https://stacks.math.columbia.edu/tag/06LK}{Tags 06LK} \& \href{https://stacks.math.columbia.edu/tag/04KF}{04KF}]{stacks-project}, it follows that 
    \begin{displaymath}
        \begin{aligned}
            \dim \mathcal{O}_{X,s\circ \overline{u}} + 1
            &= \dim \mathcal{O}_{U,u}^{sh} + 1 
            \\&= \dim \mathcal{O}_{U,u} + 1 
            \\&= \dim \mathcal{O}_{U,u^\prime}
            \\&= \dim \mathcal{O}_{U,u^\prime}^{sh}
            \\&= \dim \mathcal{O}_{X,s\circ \overline{u}^\prime}.
        \end{aligned}
    \end{displaymath}
    By the first claim, there exists an immediate generalization $v \rightsquigarrow v^\prime$ in $|U|$ such that $s(u) = s(v)$ and $s(v^\prime) = q$.
    Choose geometric points $\overline{v}\colon \operatorname{Spec}(k)\to U$ and $\overline{v}^\prime \colon \operatorname{Spec}(k^\prime) \to U$ respectively for $v$ and $v^\prime$. 
    By \cite[Lemma 2.4]{GuisadoVillaalgordo/Lank:2026}, $s \circ \overline{v}$ and $s \circ \overline{u}$ represent the same point in $|X|$. 
    Once again, from \cite[\href{https://stacks.math.columbia.edu/tag/06LK}{Tags 06LK} \& \href{https://stacks.math.columbia.edu/tag/04KF}{04KF}]{stacks-project}, 
    \begin{displaymath}
            \dim \mathcal{O}_{X,s\circ \overline{u}} + 1
            = \dim \mathcal{O}_{X,s\circ \overline{v}} + 1
            = \dim \mathcal{O}_{X,s\circ \overline{v}^\prime}.
    \end{displaymath}
    Since $q\not = s(u^\prime)$, then a similar argument above yields
    \begin{displaymath}
            \dim \mathcal{O}_{X,s\circ \overline{v}^\prime} + 1
            \leq \dim \mathcal{O}_{X,s\circ \overline{u}^\prime}.
    \end{displaymath}
    However, we obtain a contradiction,
    \begin{displaymath}
        \dim \mathcal{O}_{X,s\circ \overline{u}} + 1 < \dim \mathcal{O}_{X,s\circ \overline{u}} + 2 \leq \dim \mathcal{O}_{X,s\circ \overline{u}^\prime}.
    \end{displaymath}
    This completes the proof.
\end{proof}

\begin{lemma}
    \label{lem:wc_across_z_via_etale_presentation}
    Let $X$ be a Noetherian algebraic space admitting a dualizing complex and $\phi$ be a Thomason filtration on $X$. 
    Then $\phi$ satisfies the weak Cousin condition if, and only if, $s^{-1}\phi$ is weak Cousin for every \'{e}tale presentation $s\colon U \to X$.
\end{lemma}

\begin{proof}
    First, let $\phi$ satisfy the weak Cousin condition. 
    Consider an \'{e}tale presentation $s\colon U \to X$. 
    Recall that $s^{-1}\phi$ is defined by the assignment $n\mapsto s^{-1}(\phi(n))$. 
    Choose an immediate generalization $u\rightsquigarrow u^\prime$ in $|U|$ where $u^\prime \in s^{-1}(\phi(n))$. 
    By \Cref{lem:lifting_descending_immediate_generalizations}, $s(u)\rightsquigarrow s(u^\prime)$ is an immediate specialization. 
    Moreover, we know that $s(u^\prime) \in \phi(n)$. 
    Since $\phi$ satisfies the weak Cousin condition, it follows that $s(u)\in \phi(n-1)$. 
    Hence, $u\in s^{-1}(\phi(n-1))$, which shows $s^{-1}\phi$ satisfies weak Cousin condition.

    We prove the converse direction. 
    Suppose that $s^{-1}\phi$ satisfies the weak Cousin condition for every \'{e}tale presentation $s\colon U \to X$. 
    Fix an \'{e}tale presentation $s\colon U \to X$. 
    Let $p\rightsquigarrow p^\prime$ be an immediate generalization in $|X|$ where $p^\prime \in \phi(n)$. 
    By \Cref{lem:lifting_descending_immediate_generalizations}, there exists an immediate generalization $u \rightsquigarrow u^\prime$ such that $s(u)=p$, $s(u^\prime)=p^\prime$. 
    It follows that $u^\prime \in s^{-1}(\phi(n))$. 
    Since $s^{-1}\phi$ satisfies the weak Cousin condition, we know that $u\in s^{-1}(\phi(n-1))$. 
    Hence, $p\in \phi(n-1)$, which finishes the proof.
\end{proof}

\begin{remark}
    \label{rmk:smooth_presentation_CM-excellent}
    Let $X$ be a Noetherian algebraic space admitting a dualizing complex. 
    By \cite[\href{https://stacks.math.columbia.edu/tag/0E4X}{Tag 0E4X}]{stacks-project}, for any \'{e}tale presentation $s\colon U \to X$ by a scheme, $U$ admits a dualizing complex. 
    Then, in the sense of \cite{Cesnavicius:2021}, $U$ is CM-excellent. 
    This follows from \cite[Example 4.6]{Clark/Lank/ManaliRahul/Parker:2024}, which builds from the affine setting of \cite[Corollary 1.4]{Kawasaki:2002}.
\end{remark}

\begin{proposition}
    \label{prop:weak_Cousin_across_Z_implies_restriction}
    Let $X$ be a Noetherian algebraic space admitting a dualizing complex and $\phi$ be a Thomason filtration on $X$. 
    If $\phi$ satisfies the weak Cousin condition, then $\mathcal{U}_\phi$ restricts to $D^b_{\operatorname{coh}}(X)$.
\end{proposition}

\begin{proof}
    Let $s\colon U \to X$ be an \'{e}tale presentation from an affine scheme. 
    Choose $E\in D^b_{\operatorname{coh}}(X)$. 
    Consider the truncation triangle with respect to $\mathcal{U}_\phi$ on $D_{\operatorname{qc}}(X)$,
    \begin{displaymath}
        \tau^{\leq 0} (E) \to E \to \tau^{\geq 1} (E) \to \tau^{\leq 0} (E)[1].
    \end{displaymath}
    By \Cref{prop:t_exactness_flat_pullback}, $\mathbf{L}s^\ast$ yields a $t$-exact functor 
    \begin{displaymath}
        (D_{\operatorname{qc}}(X) , \mathcal{U}_\phi) \to (D_{\operatorname{qc}}(U) , \mathcal{U}_{s^{-1}\phi}).
    \end{displaymath} 
    Hence, the truncation triangle for $\mathbf{L}s^\ast E$ with respect to $\mathcal{U}_{s^{-1}\phi}$ is given by
    \begin{displaymath}
        \mathbf{L}s^\ast \tau^{\leq 0} (E) \to \mathbf{L}s^\ast E \to \mathbf{L}s^\ast \tau^{\geq 1} (E) \to \mathbf{L}s^\ast \tau^{\leq 0} (E)[1].
    \end{displaymath}
    Applying \Cref{lem:wc_across_z_via_etale_presentation}, $s^{-1}\phi$ satisfies the weak Cousin condition.
    Coupling \Cref{rmk:smooth_presentation_CM-excellent} with \cite[Theorem 3.5]{Takahashi:2023}, it follows that $\mathcal{U}_{s^{-1}\phi}$ restricts to $D^b_{\operatorname{coh}}(U)$. 
    In other words, $\mathbf{L}s^\ast \tau^{\leq 0} (E)$ and $\mathbf{L}s^\ast \tau^{\geq 1} (E)$ have bounded cohomology. 
    Then \cite[Lemma 2.10]{GuisadoVillaalgordo/Lank:2026} implies $\tau^{\leq 0} (E)$ and $\tau^{\geq 1} (E)$ have bounded cohomology. 
    Moreover, from flatness of $s$, \cite[\href{https://stacks.math.columbia.edu/tag/07UB}{Tag 07UB}]{stacks-project} implies $\tau^{\leq 0} (E)$ and $\tau^{\geq 1} (E)$ have coherent cohomology. 
    Thus, $\mathcal{U}_\phi$ restricts to $D^b_{\operatorname{coh}}(X)$.
\end{proof}

\begin{proof}
    [Proof of \Cref{prop:etale_verdier_localization}]
    By Zariski's Main Theorem \cite[\href{https://stacks.math.columbia.edu/tag/082K}{Tag 082K}]{stacks-project}, there exists a finite morphism $t\colon X^\prime \to X$ and an open immersion $p\colon U \to X^\prime$ such that $t\circ p = s$. 
    Consider the following diagram 
    \begin{displaymath}
        \begin{tikzcd}
            U &&& \\
            & {U\times_X U} & {X^\prime\times_X U} & U \\
            & {U\times_X X^\prime} & {X^\prime \times_X X^\prime} & {X^\prime} \\
            & U & {X^\prime} & X
            \arrow["{\Delta_s}", from=1-1, to=2-2]
            \arrow["{1_U}", bend right = -18pt, from=1-1, to=2-4]
            \arrow["{1_U}"',  bend right = 18pt, from=1-1, to=4-2]
            \arrow["{s_2}", from=2-2, to=2-3]
            \arrow["{s_1}"', from=2-2, to=3-2]
            \arrow["{q_2}", from=2-3, to=2-4]
            \arrow["{q_1}"', from=2-3, to=3-3]
            \arrow["p"', from=2-4, to=3-4]
            \arrow["{p_2}", from=3-2, to=3-3]
            \arrow["{p_1}"', from=3-2, to=4-2]
            \arrow["{t_2}", from=3-3, to=3-4]
            \arrow["{t_1}"', from=3-3, to=4-3]
            \arrow["t"', from=3-4, to=4-4]
            \arrow["p"', from=4-2, to=4-3]
            \arrow["t"', from=4-3, to=4-4]
        \end{tikzcd}
    \end{displaymath}
    obtained by base change. 
    Since $s$ is unramified, $\Delta_s$ is an open immersion \cite[\href{https://stacks.math.columbia.edu/tag/05W1}{Tag 05W1}]{stacks-project}.
    Moreover, $p$ being an open immersion implies $s_1$ is an open immersion by base change. 
    Hence, $s_1 \circ \Delta_s$ is an open immersion. 
    Since $s$ is separated, it follows that $p$ is separated \cite[\href{https://stacks.math.columbia.edu/tag/03KR}{Tag 03KR}]{stacks-project}.
    Since $s_1 \circ \Delta_s$ is the graph of $p$, $s_1 \circ \Delta_s$ is a closed immersion \cite[\href{https://stacks.math.columbia.edu/tag/03KO}{Tag 03KO}]{stacks-project}.
    Set $r\colon V \to U \times_X X^\prime $ the associated open and closed immersion of $|\times_X X^\prime |\setminus|U|$. 
    It follows that $U \sqcup V = U\times_X X^\prime$. 
    Indeed, take the natural morphism $U \sqcup V \to U\times_X X^\prime$ induced by $s_1 \circ \Delta_s$ and $r$.
    Since $U$ and $V$ form a Zariski covering of $U\times_X X^\prime$ \cite[\href{https://stacks.math.columbia.edu/tag/02YY}{Tag 02YY}]{stacks-project}, it follows that $U \sqcup V \to U\times_X X^\prime$ is a bijective open and closed immersion \cite[\href{https://stacks.math.columbia.edu/tag/03M4}{Tag 03M4}]{stacks-project}. 
    Hence, the pair $s_1 \circ \Delta_s \colon U \to U\times_X X^\prime$ and $r\colon V \to U\times_X X^\prime$ form an elementary distinguished square. 
    By \cite[\href{https://stacks.math.columbia.edu/tag/08GG}{Tag 08GG}]{stacks-project}, there exists adjoint equivalences 
    \begin{displaymath}
        \mathbf{L}r^\ast \colon D_{\operatorname{qc},V}(U\times_X X^\prime) \leftrightarrows D_{\operatorname{qc}}(V) \colon \mathbf{R}r_\ast
    \end{displaymath}
    and
    \begin{displaymath}
        \mathbf{L}(s_1 \circ \Delta_s)^\ast \colon D_{\operatorname{qc},U}(U\times_X X^\prime) \leftrightarrows D_{\operatorname{qc}}(U) \colon \mathbf{R}(s_1 \circ \Delta_s)_\ast.
    \end{displaymath}
    Denote by $\eta$ the unit of the adjunction $\mathbf{L}(s_1 \circ \Delta_s)^\ast \dashv \mathbf{R}(s_1 \circ \Delta_s)_\ast$. 
    By \cite[Lemma A.4]{GuisadoVillaalgordo/Lank:2026}, there exists a recollement 
    \begin{displaymath}
        \begin{tikzcd}
            {D_{\operatorname{qc}}(U)} && {D_{\operatorname{qc}}(U\times_X X^\prime)} && {D_{\operatorname{qc},V}(U\times_X X^\prime)}
            \arrow["{\mathbf{R}(s_1 \circ \Delta_s)_\ast}", from=1-1, to=1-3]
            \arrow["{(s_1 \circ \Delta_s)^\times}", shift left=3, bend right = -24 pt, from=1-3, to=1-1]
            \arrow["{\mathbf{L}(s_1 \circ \Delta_s)^\ast}"', shift right=3, bend right = 24 pt, from=1-3, to=1-1]
            \arrow["{i^!_U}", from=1-3, to=1-5]
            \arrow["{i^\ast_U}", shift left=3, bend right = -24 pt, from=1-5, to=1-3]
            \arrow["{i_\ast^U}"', shift right=3, bend right = 24 pt, from=1-5, to=1-3]
        \end{tikzcd}
    \end{displaymath}
    where $i_\ast^U$ is the inclusion.
    As $s_1 \circ \Delta_s$ is an open immersion, the distinguished triangle 
    \begin{equation}
        \label{eq:unit_restriction_implies_weak_Cousin_across_Z}
        \operatorname{cone}(\eta_{\mathcal{O}_{U\times_X X^\prime}})[-1] 
        \to \mathcal{O}_{U\times_X X^\prime} 
        \to \mathbf{R}(s_1 \circ \Delta_s)_\ast \mathcal{O}_U 
        \to \operatorname{cone}(\eta_{\mathcal{O}_{U\times_X X^\prime}}).
    \end{equation}
    corresponds to the distinguished triangle associated to the recollement
    \begin{displaymath}
        (i^U_\ast \circ i^!_U)(\mathcal{O}_{U\times_X X^\prime})
        \to \mathcal{O}_{U\times_X X^\prime} 
        \to \mathbf{R}(s_1 \circ \Delta_s)_\ast \mathcal{O}_U
        \to (i^U_\ast \circ i^!_U)(\mathcal{O}_{U\times_X X^\prime})[1].
    \end{displaymath}
    Hence, there exists an isomorphism
    \begin{displaymath}
        \operatorname{cone}(\eta_{\mathcal{O}_{U\times_X X^\prime}}) [-1] \cong (i^U_\ast \circ i^!_U)(\mathcal{O}_{U\times_X X^\prime}).
    \end{displaymath}
    Once more, from \cite[Lemma A.4]{GuisadoVillaalgordo/Lank:2026}, there exists a recollement 
    \begin{displaymath}
        \begin{tikzcd}
            {D_{\operatorname{qc}}(V)} && {D_{\operatorname{qc}}(U\times_X X^\prime)} && {D_{\operatorname{qc},U}(U\times_X X^\prime)}
            \arrow["{\mathbf{R}r_\ast}", from=1-1, to=1-3]
            \arrow["{r^\times}", shift left=3, bend right = -24 pt, from=1-3, to=1-1]
            \arrow["{\mathbf{L}r^\ast}"', shift right=3, bend right = 24 pt, from=1-3, to=1-1]
            \arrow["{i^!_V}", from=1-3, to=1-5]
            \arrow["{i^\ast_V}", shift left=3, bend right = -24 pt, from=1-5, to=1-3]
            \arrow["{i_\ast^V}"', shift right=3, bend right = 24 pt, from=1-5, to=1-3]
        \end{tikzcd}
    \end{displaymath}
    where $i^\ast_V$ is the natural inclusion. 
    This ensures that 
    \begin{displaymath}
        D_{\operatorname{qc},U}(U\times_X X^\prime) \subseteq {}^\perp (\mathbf{R}r_\ast D_{\operatorname{qc}}(V)).
    \end{displaymath}
    Note the adjoint equivalences implies
    \begin{displaymath}
        \mathbf{R}(s_1 \circ \Delta_s)_\ast \mathcal{O}_U \in D_{\operatorname{qc},U}(U\times_X X^\prime)
    \end{displaymath}
    and
    \begin{displaymath}
        \operatorname{cone}(\eta_{\mathcal{O}_{U\times_X X^\prime}}) \in D_{\operatorname{qc},V}(U\times_X X^\prime) = \mathbf{R}r_\ast D_{\operatorname{qc}}(V).
    \end{displaymath}
    In other words, we obtain that
    \begin{displaymath}
        \begin{aligned}
            \operatorname{Hom} 
            & (\mathbf{R}(s_1 \circ \Delta_s)_\ast \mathcal{O}_U, \operatorname{cone}(\eta_{\mathcal{O}_{U\times_X X^\prime}}))
            \\&= \operatorname{Hom}(i^V_\ast \mathbf{R}(s_1 \circ \Delta_s)_\ast \mathcal{O}_U, \mathbf{R}r_\ast \mathbf{L}r^\ast \operatorname{cone}(\eta_{\mathcal{O}_{U\times_X X^\prime}}))
            \\&= \operatorname{Hom}(\mathbf{R}(s_1 \circ \Delta_s)_\ast \mathcal{O}_U, i^!_V \mathbf{R}r_\ast \mathbf{L}r^\ast \operatorname{cone}(\eta_{\mathcal{O}_{U\times_X X^\prime}}))
            \\&\cong 0.
        \end{aligned}
    \end{displaymath}
    Consequently, \eqref{eq:unit_restriction_implies_weak_Cousin_across_Z} splits, and so 
    \begin{equation}
        \label{eq:decompose_graph}
        \mathcal{O}_{U\times_X X^\prime} 
        \cong \mathbf{R}(s_1 \circ \Delta_s)_\ast \mathcal{O}_U \oplus \operatorname{cone}(\eta_{\mathcal{O}_{U\times_X X^\prime}})[-1].
    \end{equation}
    
    Let $E\in D^b_{\operatorname{coh}}(U)$. 
    As $j$ is an open immersion, \Cref{lem:verdier} says there exists $E^\prime \in D^b_{\operatorname{coh}}(X^\prime)$ such that $E$ is a direct summand of $\mathbf{L}p^\ast E^\prime$. 
    Since $t$ is finite, $\mathbf{R}t_\ast E^\prime$ is in $D^b_{\operatorname{coh}}(X)$. 
    Since $s$ is flat, $\mathbf{L}s^\ast \mathbf{R}t_\ast E^\prime \in D^b_{\operatorname{coh}}(U)$. 
    Observe that 
    \begin{displaymath}
        \begin{aligned}
            \mathbf{L}s^\ast \mathbf{R}t_\ast E^\prime
            &\cong \mathbf{L}(t\circ p)^\ast \mathbf{R}t_\ast E^\prime
            \\&\cong \mathbf{R}(p_1)_\ast \mathbf{L}(t_2 \circ p_2)^\ast E^\prime && \textrm{(flat base change)}
            \\&\cong \mathbf{R}(p_1)_\ast (\mathcal{O}_{U\times_X X^\prime} \otimes^{\mathbf{L}} \mathbf{L}(t_2 \circ p_2)^\ast E^\prime).
            \end{aligned}
    \end{displaymath}
    Applying \eqref{eq:decompose_graph}, and using the derived tensor product commutes with direct sums, this complex is isomorphic to 
    \begin{displaymath}
        \begin{aligned}
            & \mathbf{R}(p_1)_\ast \big( (\mathbf{R}(s_1 \circ \Delta_s)_\ast \mathcal{O}_U \oplus \operatorname{cone}(\eta_{\mathcal{O}_{U\times_X X^\prime}})[-1]) \otimes^{\mathbf{L}} \mathbf{L}(t_2 \circ p_2)^\ast E^\prime \big)
            \\&\cong \mathbf{R}(p_1)_\ast \big( (\mathbf{R}(s_1 \circ \Delta_s)_\ast \mathcal{O}_U \otimes^{\mathbf{L}} \mathbf{L}(t_2 \circ p_2)^\ast E^\prime  ) \oplus (\operatorname{cone}(\eta_{\mathcal{O}_{U\times_X X^\prime}})[-1] \otimes^{\mathbf{L}} \mathbf{L}(t_2 \circ p_2)^\ast E^\prime ) \big)
            \\&\cong \mathbf{R}(p_1)_\ast(\mathbf{R}(s_1 \circ \Delta_s)_\ast \mathcal{O}_U \otimes^{\mathbf{L}} \mathbf{L}(t_2 \circ p_2)^\ast E^\prime  ) \oplus \mathbf{R}(p_1)_\ast (\operatorname{cone}(\eta_{\mathcal{O}_{U\times_X X^\prime}})[-1] \otimes^{\mathbf{L}} \mathbf{L}(t_2 \circ p_2)^\ast E^\prime )
        \end{aligned}
    \end{displaymath}
    We claim that $E$ is a direct summand of 
    \begin{displaymath}
        \mathbf{R}(p_1)_\ast (\mathbf{R}(s_1 \circ \Delta_s)_\ast \mathcal{O}_U \otimes^{\mathbf{L}} \mathbf{L}(t_2 \circ p_2)^\ast E^\prime  ).
    \end{displaymath}
    This follows from the computation:
    \begin{displaymath}
        \begin{aligned}
            \\&\cong \mathbf{R}(p_1)_\ast (\mathbf{R}(s_1 \circ \Delta_s)_\ast\mathbf{L}(t_2 \circ p_2 \circ s_1 \circ \Delta_s)^\ast E^\prime  ) && \textrm{(projection formula)}
            \\&\cong \mathbf{R}(p_1)_\ast (\mathbf{R}(s_1 \circ \Delta_s)_\ast\mathbf{L}(p \circ q_2 \circ s_2 \circ \Delta_s)^\ast E^\prime  ) 
            \\&\cong \mathbf{R}(p_1)_\ast (\mathbf{R}(s_1 \circ \Delta_s)_\ast\mathbf{L}p^\ast E^\prime  ) 
            \\&\cong  \mathbf{R}(p_1\circ s_1 \circ \Delta_s)_\ast\mathbf{L}p^\ast E^\prime  
            \\&\cong \mathbf{L}p^\ast E^\prime 
        \end{aligned}
    \end{displaymath}
    This completes the proof.
\end{proof}

\begin{proposition}
    \label{prop:restriction_implies_weak_Cousin_across_Z}
    Let $X$ be a Noetherian algebraic space admitting a dualizing complex and $\phi$ be a Thomason filtration on $X$. 
    If $\mathcal{U}_\phi$ restricts to $D^b_{\operatorname{coh}}(X)$, then $\phi$ satisfies the weak Cousin condition.
\end{proposition}

\begin{proof}
    Choose an \'{e}tale presentation $s\colon U \to X$ by an affine scheme. 
    Since $U$ is affine, \cite[Lemma 2.13]{GuisadoVillaalgordo/Lank:2026} implies $s$ is separated. 
    Let $E\in D^b_{\operatorname{coh}}(U)$. 
    By \Cref{prop:etale_verdier_localization}, there exist $E^\prime \in D^b_{\operatorname{coh}}(X)$ and $B\in D^b_{\operatorname{coh}}(U)$ such that $\mathbf{L}s^\ast E^\prime \cong E \oplus B$. 
    Consider the truncation triangle of $E^\prime$ with respect to $\phi$,
    \begin{displaymath}
        \tau^{\leq 0}_\phi E^\prime \to E^\prime \to \tau^{\geq 1}_\phi E^\prime \to (\tau^{\leq 0}_\phi E^\prime)[1].
    \end{displaymath}
    By \Cref{prop:t_exactness_flat_pullback}, 
    \begin{displaymath}
        \mathbf{L}s^\ast \tau^{\leq 0}_\phi E^\prime \to \mathbf{L}s^\ast E^\prime \to \mathbf{L}s^\ast \tau^{\geq 1}_\phi E^\prime \to (\mathbf{L}s^\ast \tau^{\leq 0}_\phi E^\prime)[1].
    \end{displaymath}
    is the truncation triangle of $\mathbf{L}s^\ast E^\prime$ with respect to $\mathcal{U}_{s^{-1}\phi}$.
    The hypothesis asserts that $\tau^{\leq 0}_\phi E^\prime , \tau^{\geq 1}_\phi E^\prime \in D^b_{\operatorname{coh}}(X)$. 
    Since $s$ is flat, it follows that $\mathbf{L}s^\ast \tau^{\leq 0}_\phi E^\prime , \mathbf{L}s^\ast \tau^{\geq 1}_\phi E^\prime \in D^b_{\operatorname{coh}}(U)$.
    Since $\mathbf{L}s^\ast E^\prime \cong E \oplus B$, it follows that 
    \begin{displaymath}
        \tau_{s^{-1}\phi}^{\leq 0} (\mathbf{L}s^\ast E^\prime)\cong \tau_{s^{-1}\phi}^{\leq 0} E \oplus \tau^{\leq 0}_{s^{-1}\phi} B \in D^b_{\operatorname{coh}}(U)
    \end{displaymath}
    and 
    \begin{displaymath}
        \tau_{s^{-1}\phi}^{\geq 1} (\mathbf{L}s^\ast E^\prime)\cong \tau_{s^{-1}\phi}^{\geq 1} E \oplus \tau^{\geq 1}_{s^{-1}\phi} B\in D^b_{\operatorname{coh}}(U).
    \end{displaymath}
    Thus, $\mathcal{U}_{s^{-1}\phi}$ restricts to an aisle on $D^b_{\operatorname{coh}}(U)$.
    By \cite[Theorem 3.5]{Takahashi:2023}, $s^{-1}\phi$ satisfies the weak Cousin condition. 
    By \Cref{lem:wc_across_z_via_etale_presentation}, $\phi$ satisfies the weak Cousin condition. 
\end{proof}

\begin{proof}
    [Proof of \Cref{thm:restrict_iff_weak_cousin}]
    This follows from \Cref{prop:restriction_implies_weak_Cousin_across_Z,prop:weak_Cousin_across_Z_implies_restriction}.
\end{proof}

\section{Classification}
\label{sec:classification}

\begin{lemma}
    \label{lem:aisles_based_off_support}
    Let $X$ be a Noetherian algebraic space. 
    Choose $\mathcal{S}_1,\mathcal{S}_2\subseteq D^b_{\operatorname{coh}}(X)$. 
    Then $\overline{\langle \mathcal{S}_1 \rangle}^{(-\infty,0]}_{\otimes}\subseteq \overline{\langle \mathcal{S}_2 \rangle}^{(-\infty,0]}_{\otimes}$ if one has for every $k\in\mathbb{Z}$,
    \begin{displaymath}
        \bigcup_{E\in \mathcal{S}_1} \operatorname{Supp}(\mathcal{H}^k (E))\subseteq \bigcup_{\substack{E^\prime \in \mathcal{S}_2 \\ i\geq k}} \operatorname{Supp}(\mathcal{H}^i (E^\prime)).
    \end{displaymath} 
\end{lemma}

\begin{proof}
    Choose $A\in \overline{\langle \mathcal{S}_1 \rangle}^{(-\infty,0]}_{\otimes}$. 
    Consider the truncation triangle of $A$,
    \begin{displaymath}
        \tau^{\leq 0}_{\mathcal{S}_2} A \to A \to \tau^{\geq 1}_{\mathcal{S}_2} A \to (\tau^{\leq 0}_{\mathcal{S}_2} A)[1],
    \end{displaymath}
    with respect to $\overline{\langle \mathcal{S}_2 \rangle}^{(-\infty,0]}_{\otimes}$. 
    Let $s \colon U \to X$ be an \'{e}tale presentation. 
    By \Cref{prop:t_exactness_flat_pullback}, $\mathbf{L} s^\ast$ induces $t$-exact functors,
    \begin{displaymath}
        (D_{\operatorname{qc}}(X), \overline{\langle \mathcal{S}_2 \rangle}^{(-\infty,0]}_{\otimes} ) \to (D_{\operatorname{qc}}(U), \overline{\langle \mathbf{L} s^\ast \mathcal{S}_2 \rangle}^{(-\infty,0]}_{\otimes} ).
    \end{displaymath}
    Hence, we have the truncation triangle $\mathbf{L} s^\ast A$,
    \begin{displaymath}
        \mathbf{L} s^\ast \tau^{\leq 0}_{\mathcal{S}_2} A \to \mathbf{L} s^\ast A \to \mathbf{L} s_i^\ast \tau^{\geq 1}_{\mathcal{S}_2} A \to \mathbf{L} s_i^\ast (\tau^{\leq 0}_{\mathcal{S}_2} A)[1],
    \end{displaymath}
    with respect to $\overline{\langle \mathbf{L} s^\ast \mathcal{S}_2 \rangle}^{(-\infty,0]}_{\otimes}$.
    Then our hypothesis, coupled with \cite[Proposition 5.2]{Hrbek:2020}, implies $\mathbf{L} s^\ast A\in \overline{\langle \mathbf{L} s^\ast \mathcal{S}_2 \rangle}^{(-\infty,0]}_{\otimes}$. 
    Note that we are applying \cite[\href{https://stacks.math.columbia.edu/tag/071Q}{Tags 071Q}, \href{https://stacks.math.columbia.edu/tag/08H8}{08H8}, \& \href{https://stacks.math.columbia.edu/tag/07TY}{07TY}]{stacks-project}, to transfer \cite[Proposition 5.2]{Hrbek:2020} on the small Zariski site of $U$ to the small \'{e}tale site. 
    Thus, $\tau^{\geq 1}_{\mathcal{S}_2} A= 0$, and so, $A\in \overline{\langle \mathcal{S}_2 \rangle}^{(-\infty,0]}_{\otimes}$.
\end{proof}

\begin{lemma}
    \label{lem:aisles_based_off_support_pseudo_coherent}
    Let $X$ be a Noetherian algebraic space. 
    For any $\mathcal{S}\subseteq D^b_{\operatorname{coh}}(X)$, one has $\overline{\langle \mathcal{S} \rangle}^{(-\infty,0]}_{\otimes} = \overline{\langle \mathcal{S}^\prime \rangle}^{(-\infty,0]}_{\otimes}$ where $\mathcal{S}^\prime := \{ \mathcal{H}^j (E)[-j] : S\in \mathcal{S} , j\in \mathbb{Z} \}$.
\end{lemma}

\begin{proof}
    This is argued in a similar fashion to \Cref{lem:aisles_based_off_support} where one replaces \cite[Proposition 5.2]{Hrbek:2020} with \cite[Corollary 3.8]{AlonsoTarrio/JeremiasLopez/Saorin:2010}.
\end{proof}

\begin{proposition}
    \label{prop:aisles_are_compactly_generated}
    Let $X$ be a Noetherian algebraic space. 
    If $\mathcal{S}\subseteq D^b_{\operatorname{coh}}(X)$, then $\overline{\langle \mathcal{S} \rangle}^{(-\infty,0]}_{\otimes}$ is compactly generated.
\end{proposition}

\begin{proof}
    Applying \Cref{lem:aisles_based_off_support_pseudo_coherent}, we can reduce to the case $\mathcal{S}$ consists of a single object $E$. 
    By shifting, we may impose $E$ be concentrated in degrees $[0,b]\cap \mathbb{Z}$. 
    Additionally, for each $k\in [0,b]\cap \mathbb{Z}$, we can assume $\operatorname{Supp}(\mathcal{H}^k (E))$ is irreducible.
    For each $s\in [0,b]\cap \mathbb{Z}$, let $i_s \colon Z_s \to X$ be the associated closed immersion of the reduced induced closed subspace structure $\operatorname{Supp}(\mathcal{H}^s(E))$. 
    Define $A:=\oplus^b_{s=0} (i_s)_\ast \mathcal{O}_{Z_s} [-s]$. 
    The support of $\mathcal{H}^k (A)$ coincides with that of $\mathcal{H}^k (E)$ for all $k\in \mathbb{Z}$.
    It follows that $\operatorname{Supp}(\mathcal{H}^k (A))\subseteq \cup_{i\geq k} \operatorname{Supp}(\mathcal{H}^i (E))$.

    By \Cref{lem:aisles_based_off_support}, $\overline{\langle A \rangle}^{(-\infty,0]}_{\otimes}\subseteq \overline{\langle E \rangle}^{(-\infty,0]}_{\otimes}$. 
    We check the reverse containment. From similar reasoning, $\operatorname{Supp}(\mathcal{H}^k (E))\subseteq \cup_{i\geq k} \operatorname{Supp}(\mathcal{H}^i (A))$. 
    Again, from \Cref{lem:aisles_based_off_support}, $\overline{\langle E \rangle}^{(-\infty,0]}_{\otimes}\subseteq \overline{\langle A \rangle}^{(-\infty,0]}_{\otimes}$. 
    Thus, $\overline{\langle E \rangle}^{(-\infty,0]}_{\otimes} = \overline{\langle A \rangle}^{(-\infty,0]}_{\otimes}$. 
    By \cite[\href{https://stacks.math.columbia.edu/tag/08HH}{Tag 08HH}]{stacks-project}, for each $(i_s)_\ast \mathcal{O}_{Z_s} [-s]$, there exist $P_s \in \operatorname{Perf}_{Z_s}(X)$ and morphism $P_s \xrightarrow{\alpha_s}(i_s)_\ast \mathcal{O}_{Z_s} [-s]$ such that $\mathcal{H}^t (\alpha_s)$ is an isomorphism for $t>s-1$ and $\mathcal{H}^s (\alpha_s)$ is surjective. 
    Then, from \Cref{lem:aisles_based_off_support}, $\overline{\langle \oplus^b_{s=0} P_s \rangle}^{(-\infty,0]}_{\otimes} = \overline{\langle A \rangle}^{(-\infty,0]}_{\otimes}$. 
    Consequently, from \Cref{rmk:HLP_HLLP_proxy_aisle},
    the desired claim follows.
\end{proof}

\begin{corollary}
    [cf.\ {\cite[Theorem 3.10]{AlonsoTarrio/JeremiasLopez/Saorin:2010}}]
    \label{thm:ALS_global_full_classification}
    Let $X$ be a Noetherian algebraic space. 
    Then the following are equivalent for any $\otimes$-aisle $\mathcal{A}$ on $D_{\operatorname{qc}}(X)$:
    \begin{enumerate}
        \item \label{thm:ALS_global_full_classification1} $\mathcal{A}=\mathcal{U}_{\phi}$ associated to a Thomason filtration $\phi$ on $X$
        \item \label{thm:ALS_global_full_classification2} $\mathcal{A} = \overline{\langle\mathcal{P} \rangle}^{(-\infty,0]}_{\otimes}$ for some $\mathcal{P}\subseteq \operatorname{Perf}(X)$
        \item \label{thm:ALS_global_full_classification3} $\mathcal{A} = \overline{\langle \mathcal{B} \rangle}^{(-\infty,0]}_{\otimes}$ for some $\mathcal{B}\subseteq D^b_{\operatorname{coh}}(X)$.
    \end{enumerate}
    In particular, for each case above, there is $\mathcal{Q}\subseteq \operatorname{Perf}(X)$ such that $\mathcal{A} = \overline{\langle \mathcal{Q} \rangle}^{(-\infty,0]}$ (e.g.\ no tensor needed).
\end{corollary}

\begin{proof}
    $\eqref{thm:ALS_global_full_classification1} \iff \eqref{thm:ALS_global_full_classification2}$ is \cite[Theorem 1.4]{Hrbek/Lank/Pizzirani:2025}; 
    $\eqref{thm:ALS_global_full_classification2} \implies \eqref{thm:ALS_global_full_classification3}$ is straightforward; $\eqref{thm:ALS_global_full_classification3} \implies \eqref{thm:ALS_global_full_classification2}$ is \Cref{prop:aisles_are_compactly_generated}.
\end{proof}

\begin{proof}
    [Proof of \Cref{thm:dualizing_complex_classification}]
    To start, we check that for any $\otimes$-aisle $\mathcal{A}$ on $D^b_{\operatorname{coh}}(X)$, there exists a unique Thomason filtration $\phi_{\mathcal{A}}$ on $X$ such that $\mathcal{A} = D^b_{\operatorname{coh}}(X)\cap \mathcal{U}_\phi$. 
    By \Cref{prop:aisles_are_compactly_generated}, $\overline{\langle \mathcal{A} \rangle}^{(-\infty,0]}_{\otimes}$ is compactly generated on $D_{\operatorname{qc}}(X)$. 
    Define $\phi_{\mathcal{A}}$ to be the associated Thomason filtration of $\overline{\langle \mathcal{A} \rangle}^{(-\infty,0]}_{\otimes}$ on $D_{\operatorname{qc}}(X)$. 
    Hence, $\overline{\langle \mathcal{A} \rangle}^{(-\infty,0]}_{\otimes} = \mathcal{U}_{\phi_{\mathcal{A}}}$. 

    It can be checked that $\mathcal{A} \subseteq \mathcal{U}_{\phi_{\mathcal{A}}} \cap D^b_{\operatorname{coh}}(X)$. 
    We prove the reverse inclusion. Choose $E\in \mathcal{U}_{\phi_{\mathcal{A}}} \cap D^b_{\operatorname{coh}}(X)$. There exists the truncation triangle of $E$ with respect to $\mathcal{A}$,
    \begin{displaymath}
        \tau^{\leq 0} E \to E \to \tau^{\geq 1} E \to (\tau^{\leq 0} E)[1].
    \end{displaymath}
    Note that this distinguished triangle is the truncation triangle for $E$ with respect to $\mathcal{U}_{\phi_{\mathcal{A}}}$ (see e.g.\ proof of \Cref{lem:dbcoh_tensor_via_perfect_tensor}). 
    As $E\in \mathcal{U}_{\phi_{\mathcal{A}}}$, we obtain $\tau^{\leq 0} E \to E$ is an isomorphism. 
    Thus, $\mathcal{A} = \mathcal{U}_{\phi_{\mathcal{A}}} \cap D^b_{\operatorname{coh}}(X)$.

    Given a pair of distinct $\otimes$-aisles $\mathcal{A}$ and $\mathcal{B}$ on $D^b_{\operatorname{coh}}(X)$, we must have that $\mathcal{U}_{\phi_{\mathcal{A}}} \not= \mathcal{U}_{\phi_{\mathcal{B}}}$. 
    To see, observe that if $\mathcal{U}_{\phi_{\mathcal{A}}} = \mathcal{U}_{\phi_{\mathcal{B}}}$, one has
    \begin{displaymath}
        \mathcal{A} = \mathcal{U}_{\phi_{\mathcal{A}}} \cap D^b_{\operatorname{coh}}(X) = \mathcal{U}_{\phi_{\mathcal{B}}} \cap D^b_{\operatorname{coh}}(X) = \mathcal{B}.
    \end{displaymath}
    Yet, this is absurd. 
    Hence, the assignment $\mathcal{A} \to \phi_{\mathcal{A}}$ is injective. 
    By \Cref{thm:restrict_iff_weak_cousin}, $\phi_{\mathcal{A}}$ satisfies the weak Cousin condition.
    Conversely, given any Thomason filtration $\phi$ on $X$ satisfying the weak Cousin condition, \Cref{thm:restrict_iff_weak_cousin} yields that $\mathcal{U}_\phi$ is an $\otimes$-aisle which restricts to $D^b_{\operatorname{coh}}$. 
    Therefore, the classification in \cite[Theorem 1.4]{Hrbek/Lank/Pizzirani:2025} induces the desired bijection.
\end{proof}

\begin{appendix}

\section{Verdier localizations}
\label{app:verdier_loc}

We record space theoretic versions of \cite[Appendix A]{Hall:2022}, \cite[Appendix A]{Lank/Peng:2025}, and \cite[Appendix B]{Hall/Lamarche/Lank/Peng:2025}.
While these are straightforward analogs, the details on the small \'{e}tale site are spelled out for convenience.

\begin{notation}
    Let $X$ be a Noetherian algebraic space. 
    We write $\overline{\operatorname{coh}}(X)\subseteq\operatorname{Qcoh}(X)$ for the strictly full subcategory consisting of objects $E$ such that $\Gamma(\operatorname{Spec}(R),E)$ is countably generation $R$-modules for all \'{e}tale morphisms $\operatorname{Spec}(R)\to X$. 
\end{notation}

\begin{remark}
    \label{rmk:countable_serre_subcat}
    Let $X$ be a Noetherian algebraic space. 
    It is clear that every zero object is in $\overline{\operatorname{coh}}(X)$.
    Moreover, every quotient of an object from $\overline{\operatorname{coh}}(X)$ also satisfies having countable sections over affine schemes \'{e}tale over $X$. 
    Furthermore, it can be verified that subobjects of objects from $\overline{\operatorname{coh}}(X)$ satisfying having countable sections over affine schemes \'{e}tale over $X$ (see e.g.\ \cite[Lemma A.1]{Hall:2022}).
    Additionally, it is straightforward to check that $\overline{\operatorname{coh}}(X)$ is closed under extensions. 
    Applying \cite[\href{https://stacks.math.columbia.edu/tag/02MP}{Tags 02MP} \& \href{https://stacks.math.columbia.edu/tag/03M1}{03M1}]{stacks-project}, it follows that $\overline{\operatorname{coh}}(X)$ is a Serre subcategory of $\operatorname{Mod}(X)$.
    Then, by \cite[\href{https://stacks.math.columbia.edu/tag/06UQ}{Tag 06UQ}]{stacks-project}, $D_{\overline{\operatorname{coh}}(X)}(X)$ is a triangulated subcategory of $D(X)$.
    In fact, since objects of $D_{\overline{\operatorname{coh}}(X)}(X)$ have quasi-coherent cohomology, it is a triangulated subcategory of $D_{\operatorname{qc}}(X)$.
    Define $D^b_{\overline{\operatorname{coh}}}(X) := D^b_{\operatorname{qc}}(X) \cap D_{\overline{\operatorname{coh}}(X)}(X)$.
\end{remark}

\begin{lemma}
    \label{lem:halla3}
    Let $X$ be a Noetherian algebraic space. 
    Fix some $a\leq b$ in $\mathbb{Z}$.
    Consider $E\in D_{\operatorname{qc}}(X)$ such that $\mathcal{H}^i (E)\cong 0$ for all $i\not\in [a,b]$ and $\mathcal{H}^i (E)$ has countable sections over affine schemes \'{e}tale over $X$.
    Then $E\cong \operatorname{hocolim}_s E_s$ for some $E_s\in D_{\operatorname{coh}}(X)$ such that $\mathcal{H}^i (E_s)\cong 0$ for all $i\not\in [a,b]$ and $\mathcal{H}^i (E_s)$ has countable sections over affine schemes \'{e}tale over $X$.
\end{lemma}

\begin{proof}
    Recall that for any $E\in \operatorname{Qcoh}(X)$ which has countable sections over affine schemes \'{e}tale over $X$, $E\cong \operatorname{hocolim}_s E_s$ in $D_{\operatorname{qc}}(X)$ where $E_s \in \operatorname{coh}(X)$ and $E_s \subseteq E$.
    To see, use that $E$ is a filtered colimit of its coherent subsheaves \cite[\href{https://stacks.math.columbia.edu/tag/07UV}{Tag 07UV}]{stacks-project}, and that filtered colimits are exact. 
    Since $E$ has countable sections over affine schemes \'{e}tale over $X$, one can refine this filtered colimit to a countable filtered colimit. 
    We prove the desired claim by induction on $n:= |b-a|\geq 0$. 

    By shifting, we can impose that $b=0$. 
    There exists nothing to show in the base case. 
    Applying the observation above with the inductive hypothesis, it follows that 
    $\mathcal{H}^0 (E) \cong \operatorname{hocolim}_s E^0_s$ and $\tau^{\leq -1} E \cong \operatorname{hocolim}_s F_s$ where $E^0_s\in \operatorname{coh}(X)$ and $F_s \in D_{\operatorname{qc}}(X)$ such that $\mathcal{H}^i (F_s)$ are coherent, and $\mathcal{H}^i (F_s)\cong 0$ for all $i\in [a,-1]$. 
    We have a distinguished triangle
    \begin{displaymath}
        \tau^{\leq -1} E \to E \to \mathcal{H}^0 (E)[0] \xrightarrow{\theta} (\tau^{\leq -1} E)[1].
    \end{displaymath}
    We obtain the desired $E_s$ by induction on $s$. 
    For each $s$, we obtain an induced morphism $\theta_s \colon E^0_s \to \mathcal{H}^0 (E)[0] \xrightarrow{\theta} (\tau^{\leq -1} E)[1]$. Since $F_s$ is concentrated in degrees $[a,-1]$ and $E^0_s$ are coherent, there exists $r_s$ such that $\theta_s$ factors through $\theta^\prime_s \colon E^0_s \to F_{r_s}[1]$. If needed, we can impose $s\geq s^\prime$ so that $r_s \geq r_{s^\prime}$. Set $E_s$ to be the cone of $\theta^\prime_s$. 
    This yields a morphism of distinguished triangles,
    \begin{displaymath}
        \begin{tikzcd}
            {F_{r_s}} & {E_s} & {E^0_s} & {F_{r_s}[1]} \\
            {F_{r_{s+1}}} & {E_{s+1}} & {E^0_{s+1}} & {F_{r_{s+1}}[1].}
            \arrow[from=1-1, to=1-2]
            \arrow[from=1-1, to=2-1]
            \arrow[from=1-2, to=1-3]
            \arrow["\exists"', dashed, from=1-2, to=2-2]
            \arrow[from=1-3, to=1-4]
            \arrow[from=1-3, to=2-3]
            \arrow[from=1-4, to=2-4]
            \arrow[from=2-1, to=2-2]
            \arrow[from=2-2, to=2-3]
            \arrow[from=2-3, to=2-4]
        \end{tikzcd}
    \end{displaymath}
    The morphism $e_s$ can be chosen so that the morphism of distinguished triangles is good in the sense of \cite[Definition 1.3.14]{Neeman:2001}. 
    Arguing like \cite{MathOverflow:2026}, we get a distinguished triangle 
    \begin{displaymath}
        \tau^{\leq -1}E \to \operatorname{hocolim}_s E_s \to \mathcal{H}^0 (E)[0] \xrightarrow{\theta} (\tau^{\leq -1} E)[1].
    \end{displaymath}
    Consequently, $\operatorname{hocolim}_s E_s \cong E$, which completes the proof.
\end{proof}

\begin{lemma}
    \label{lem:verdier}
    Let $X$ be a Noetherian algebraic space. 
    Suppose $j\colon U\to X$ is an open immersion. 
    Then there exists a Verdier localization,
    \begin{displaymath}
        D^b_{\operatorname{coh},|X| \setminus |U|} (X) \to D^b_{\operatorname{coh}}(X) \xrightarrow{\mathbf{L} j^\ast} D^b_{\operatorname{coh}}(U).
    \end{displaymath}
\end{lemma}

\begin{proof}
    Set $Z := |X| \setminus |U|$. 
    First, we claim that $\mathbf{R}j_\ast D^b_{\overline{\operatorname{coh}}}(U)\subseteq D^b_{\overline{\operatorname{coh}}}(X)$. 
    Fix $E \in D^b_{\overline{\operatorname{coh}}}(U)$. 
    Choose an \'{e}tale morphism $s\colon V \to X$ from an affine scheme. 
    Consider the fibered square
    \begin{displaymath}
        \begin{tikzcd}
            {U\times_X V} & V \\
            U & {X.}
            \arrow["{j^\prime}", from=1-1, to=1-2]
            \arrow["{s^\prime}"', from=1-1, to=2-1]
            \arrow["s", from=1-2, to=2-2]
            \arrow["j"', from=2-1, to=2-2]
        \end{tikzcd}
    \end{displaymath}  
    Pick an affine open cover $U\times_X V = \cup^n_{i=1} U_i$. 
    Now, if $E \in D^b_{\overline{\operatorname{coh}}}(U)$, then the cohomology sheaves of $E$ have sets of countable sections over each $U_i$. 
    We can use \cite[\href{https://stacks.math.columbia.edu/tag/08GH}{Tag 08GH}]{stacks-project} to translate the problem onto the small Zariski sites of these schemes. 
    Define $U^\prime := \cup^n_{i=2} U_i$. 
    Applying \cite[\href{https://stacks.math.columbia.edu/tag/08BX}{Tags 08BX},  \href{https://stacks.math.columbia.edu/tag/08BV}{08BV}, \& \href{https://stacks.math.columbia.edu/tag/01EZ}{01EZ}]{stacks-project}, there exists the distinguished triangle in $D(\mathcal{O}_V (V))$,
    \begin{displaymath}
        \mathbf{R}\Gamma(U\times_X V,E) \to \mathbf{R}\Gamma(U_1,E) \oplus \mathbf{R}\Gamma(U^\prime,E) \to \mathbf{R}\Gamma(U_1 \cap U^\prime,E)\to \mathbf{R}\Gamma(U\times_X V,E)[1].
    \end{displaymath}
    Then the claim follows by an inductive argument on $n$. 
    Alternatively, this could be argued using a C\v{e}ch complex (see proof of \cite[Proposition B.1]{Hall/Lamarche/Lank/Peng:2025}). 
    
    Next, we check that $\mathbf{L}j^\ast D^b_{\overline{\operatorname{coh}}}(X) \subseteq D^b_{\overline{\operatorname{coh}}}(U)$.
    Recall that any object of $D^b_{\overline{\operatorname{coh}}}(X)$ can be obtained from its cohomology sheaves using finitely many cones, shifts, finite coproducts, and direct summands (see \cite[Lemma 3.14]{GuisadoVillaalgordo/Lank:2026}). 
    Hence, to show that $\mathbf{L}j^\ast D^b_{\overline{\operatorname{coh}}}(X) \subseteq D^b_{\overline{\operatorname{coh}}}(U)$, it suffices to verify that $\mathbf{L}j^\ast \overline{\operatorname{coh}}(X) \subseteq D^b_{\overline{\operatorname{coh}}}(U)$. 
    Let $E\in \overline{\operatorname{coh}}(X)$. 
    Consider an \'{e}tale morphism $t\colon \operatorname{Spec}(R) \to U$. 
    Then $j\circ t$ is \'{e}tale with affine source. 
    Thus, $\Gamma(\operatorname{Spec}(R),E)$ is a countably generated $R$-module, and the claim follows.

    Now, we return to the proof. 
    By \cite[Lemma A.4]{GuisadoVillaalgordo/Lank:2026}, the counit $\mathbf{L}j^\ast \mathbf{R}j_\ast \to 1$ is an isomorphism on $D_{\operatorname{qc}}$.
    As $\mathbf{R}j_\ast D^b_{\overline{\operatorname{coh}}}(U) \subseteq D^b_{\overline{\operatorname{coh}}}(X)$ 
    there exists a Verdier localization,
    \begin{displaymath}
        D^b_{\overline{\operatorname{coh}},|Z|} (X) \to D^b_{\overline{\operatorname{coh}}}(X) \xrightarrow{\mathbf{L} j^\ast} D^b_{\overline{\operatorname{coh}}}(U).
    \end{displaymath}
    See e.g.\ \cite[Lemma 3.4]{Hall/Rydh:2017}. 
    Observe that 
    \begin{displaymath}
        D^b_{\operatorname{coh},|Z|} (X) = D^b_{\overline{\operatorname{coh}},|Z|} (X) \cap D^b_{\operatorname{coh}} (X).
    \end{displaymath}
    Thus, from the definition of a Verdier localization, it suffices to prove the following:
    \begin{enumerate}
        \item \label{lem:verdier1} If $M \in D^b_{\operatorname{coh}}(U)$, then there exists $\overline{M} \in D^b_{\operatorname{coh}}(X)$ such that $\mathbf{L} j^\ast\overline{M} \simeq M$ in $D^b_{\operatorname{coh}}(U)$.
        \item \label{lem:verdier2} If $\phi \colon E \to F$ is a morphism in $D^b_{\overline{\operatorname{coh}}}(X)$ with $\mathbf{L} j^\ast \phi$ an isomorphism and $F \in D^b_{\operatorname{coh}}(X)$, then there exists a map $\psi \colon E^\prime \to E$ in $D^b_{\overline{\operatorname{coh}}}(X)$ with $\mathbf{L} j^\ast \psi$ an isomorphism and $E^\prime \in D^b_{\operatorname{coh}}(X)$.
    \end{enumerate}

    We start with \eqref{lem:verdier1}.
    Recall that we have shown that $\mathbf{R}j_\ast M \in D^b_{\overline{\operatorname{coh}}}(X)$. 
    By \Cref{lem:halla3}, it follows that $\mathbf{R}j_\ast M \cong \operatorname{hocolim}_s M_s$, where $M_s \in D^b_{\operatorname{coh}}(X)$ and the induced maps $\mathcal{H}^k(M_s) \to \mathcal{H}^k(\mathbf{R} j_\ast M)$ are monomorphisms for all $k \in \mathbf{Z}$. 
    As $j$ is an open immersion, it is flat and, so the induced maps $\mathbf{L} j^\ast M_s \to \mathbf{L} j^\ast \mathbf{R} j_\ast M \simeq M$ induce monomorphisms $\mathcal{H}^k(\mathbf{L} j^\ast M_s) \to \mathcal{H}^k(M)$ for all $k\in \mathbf{Z}$. 
    However, $M \in D^b_{\operatorname{coh}}(U)$, and so we may take $s \gg 0$ such that $\mathbf{L} j^\ast M_s \simeq M$, which gives \eqref{lem:verdier1}. 
    
    Lastly, we check \eqref{lem:verdier2}.
    Consider the distinguished triangle:
    \begin{displaymath}
        G \to E \xrightarrow{\phi} F \to G[1].
    \end{displaymath}
    Then $\mathbf{L} j^\ast G \simeq 0$ and so $G \in D^b_{\overline{\operatorname{coh}},|Z|}(X)$. 
    By \Cref{lem:halla3}, $G \simeq \operatorname{hocolim}_s G_s$ where $G_s \in D^b_{\operatorname{coh}}(X)$ and $\mathcal{H}^k(G_s) \to \mathcal{H}^k(G)$ a monomorphism for all $k\in \mathbf{Z}$. 
    Hence, $\mathbf{L} j^\ast G_s \simeq 0$ for all $s$, and so $G_s \in D^b_{\operatorname{coh},|Z|}(X)$. 
    As $F\in D^b_{\operatorname{coh}}(X)$, $F \to G[1]$ factors through $G_s[1] \to G[1]$ for some $s \gg 0$. 
    There exists an induced morphism of distinguished triangles
    \begin{displaymath}
        \begin{tikzcd}
            {G_s} & {E^\prime} & F & {G_s [1]} \\
            G & E & F & {G[1]}
            \arrow[from=1-1, to=1-2]
            \arrow[from=1-1, to=2-1]
            \arrow[from=1-2, to=1-3]
            \arrow[from=1-2, to=2-2]
            \arrow[from=1-3, to=1-4]
            \arrow["{=}"', from=1-3, to=2-3]
            \arrow[from=1-4, to=2-4]
            \arrow[from=2-1, to=2-2]
            \arrow[from=2-2, to=2-3]
            \arrow[from=2-3, to=2-4]
        \end{tikzcd}
    \end{displaymath}
    which finishes the proof.
\end{proof}

\begin{lemma}
    \label{lem:relative_denseness_open_immersion}
    Let $X$ be a Noetherian algebraic space. 
    Consider an open immersion $j\colon U \to X$. 
    Suppose $Z\subseteq |U|$ is closed. 
    Denote by $\overline{Z}$ for the closure of $Z$ in $|X|$. 
    Then $\mathbf{L} j^\ast$ restricts to a Verdier localization $D^b_{\operatorname{coh},\overline{Z}} (X) \to D^b_{\operatorname{coh},Z}(U)$ whose kernel consists of $E\in D^b_{\operatorname{coh},\overline{Z}} (X)$ with support is contained in $Z\cap (|X|\setminus |U|)$.
\end{lemma}

\begin{proof}
    This follows from the proof of \Cref{lem:verdier} after paying attention to support considerations. 
    To see, observe that the support constraints hold for \eqref{lem:verdier1} and \eqref{lem:verdier2} in the proof of \Cref{lem:verdier}. 
    In particular, in \eqref{lem:verdier1}, say we have $\mathbf{R} j_\ast M \cong \operatorname{hocolim}_s M_s$ for $M\in D^b_{\operatorname{coh},Z}(U)$. 
    By \cite[Lemma A.3]{GuisadoVillaalgordo/Lank:2026}, $\mathbf{R}f_\ast M \in D^b_{\operatorname{coh},\overline{Z}}(X)$, and hence $M_s \in D^b_{\operatorname{coh},\overline{Z}}(X)$ for all $s$. 
    Then it follows that $\mathbf{L}j^\ast M_s \in  D^b_{\operatorname{coh},Z}(U)$. 
    On the other hand, cones of morphisms of objects in $D^b_{\operatorname{coh},\overline{Z}}(X)$ belong to $D^b_{\operatorname{coh},\overline{Z}}(X)$. 
    Hence, $G$ in \eqref{lem:verdier2} belongs $D^b_{\operatorname{coh},\overline{Z}}(X)$.
\end{proof}

\end{appendix}

\bibliographystyle{alpha}
\bibliography{mainbib}

\end{document}